\documentclass{article}
\usepackage[protrusion=false,expansion,final,tracking=true,kerning=true,spacing=true,factor=1100,stretch=10,shrink=10]{microtype}
\usepackage[utf8]{inputenc}
\usepackage{amsmath}
\usepackage{amsthm}
\usepackage{amsfonts}
\usepackage{amssymb}
\usepackage{bm}
\usepackage{graphicx}
\usepackage{hyperref}
\usepackage{color}
\usepackage{framed}
\usepackage[a4paper, margin=1in]{geometry}
\usepackage{enumitem}
\usepackage[most]{tcolorbox}
\usepackage{ytableau}
\usepackage[dvipsnames]{xcolor}
\usepackage{multirow}
\usepackage{stmaryrd}
\usepackage{mathtools}
\usepackage{thmtools}
\usepackage{blkarray}
\usepackage{soul}
\usepackage{xparse}
\usepackage{tikz}
\usepackage[framemethod=TikZ]{mdframed}
\usepackage{float}
\usetikzlibrary{calc}
\usetikzlibrary{positioning}

\sethlcolor{black!10!white}

\declaretheoremstyle[mdframed={linecolor=gray,linewidth=0.5pt}]{examplestyle}
\theoremstyle{plain}
\declaretheorem[name=Theorem, numberwithin=section]{theorem}
\declaretheorem[name=Conjecture, sibling=theorem]{conjecture}
\declaretheorem[name=Proposition, sibling=theorem]{proposition}
\declaretheorem[name=Lemma, sibling=theorem]{lemma}
\declaretheorem[name=Example, sibling=theorem, style=examplestyle]{example}
\declaretheorem[name=Definition, sibling=theorem]{definition}
\declaretheorem[name=Remark, sibling=theorem]{remark}

\newcounter{ybpos}
\newcounter{xbpos}
\newlength{\boxsize}

\NewDocumentEnvironment{YD}{sO{15pt}}{%
    \setcounter{ybpos}{0}
    \setlength{\boxsize}{#2}
    \begin{tikzpicture}[outer sep=0pt]
}{%
    \end{tikzpicture}
}

\NewDocumentCommand{\boxrow}{m}{
    \setcounter{xbpos}{0}
    \draw[] (0,-\theybpos*\boxsize) -- (0,-\theybpos*\boxsize+\boxsize);
    \foreach \xbwd/\boxcolor/\bsymb in {#1}{%
    	\foreach \i in {1,...,\xbwd} {
	        \node[draw,
	            fill=\boxcolor,
	            minimum height=\boxsize,
	            minimum width=\boxsize,
	            anchor=south west] at (\thexbpos*\boxsize,-\theybpos*\boxsize) {\(\bsymb\)};
	        \addtocounter{xbpos}{1}
		}
    }
    \stepcounter{ybpos}%
}

\NewDocumentCommand{\boxrowa}{m}{
    \setcounter{xbpos}{0}
    \draw[] (0,-\theybpos*\boxsize) -- (0,-\theybpos*\boxsize+\boxsize) node[midway]{\(\ast\)};
    \stepcounter{ybpos}%
}

\NewDocumentCommand{\boxrowta}{m}{
    \setcounter{xbpos}{0}
    \foreach \xbwd/\boxcolor/\bsymb in {#1}{%
    	\foreach \i in {1,...,\xbwd} {
	        \node[draw,
	            fill=\boxcolor,
	            minimum height=\boxsize,
	            minimum width=\boxsize,
	            anchor=south west] at (\thexbpos*\boxsize,-\theybpos*\boxsize) {\(\bsymb\)};
	        \addtocounter{xbpos}{1}
		}
    }
    \draw[] (\thexbpos*\boxsize,-\theybpos*\boxsize+\boxsize) -- (\thexbpos*\boxsize+\boxsize,-\theybpos*\boxsize+\boxsize) node[midway]{\(\ast\)};
    \addtocounter{xbpos}{1}
    \stepcounter{ybpos}%
}

\title{\scshape An iterative-bijective approach to asymmetric generalizations of Schur's theorem}
\author{Laure Velenik}
\date{\today}

\begin{document}
\maketitle
\abstract{In this paper, we present a new Rogers--Ramanujan type identity for overpartitions by extending the asymmetrical version of Schur’s theorem due to Lovejoy to a broader class of infinite products. More precisely, we provide a combinatorial interpretation of the following product, for any positive integer $k$, as a generating function for a class of overpartitions in which parts appear in $2^k - 1$ colors:
\[
\frac{(-y_1 q;q)_\infty \cdots (-y_k q;q)_\infty}{(y_1 d q;q)_\infty}.
\]
Our proof is bijective and unifies two earlier approaches: Lovejoy’s bijective proof of the asymmetrical Schur theorem and the iterative-bijective technique developed by Corteel and Lovejoy.}

\section{Introduction}\label{sec:introduction}

\begin{definition}
Let $n\in\mathbb{N}.$ Let $\lambda_1 \geq \lambda_2 \geq\dots\geq \lambda_L$ be positive integers called parts, such that $\lambda_1 + \lambda_2 + \dots + \lambda_L = n.$ Then \( \lambda=(\lambda_1,\lambda_2,\dots,\lambda_L) \) is called an integer partition of $n.$
\end{definition}
\begin{remark}
There is exactly one partition of $0,$ that is the empty partition.
\end{remark}
Let \(p(n)\) denote the number of integer partitions of \(n\). The generating function for integer partitions is
\[
    \sum_{n\geq0} p(n) q^n = \frac{1}{(q;q)_\infty},
\]
where we use the q-Pochhammer symbol defined as
\[
(a;q)_N\coloneqq \prod_{n = 0}^{N-1} (1-aq^n),
\]
where $N\in\mathbb{N}\cup\{\infty\}$ and we use the convention that $(a;q)_0 \coloneqq 1.$
\medskip

In 1926, Schur~\cite{Schur1926} proved a now famous theorem on integer partitions:
\begin{definition}
Let $n\in\mathbb{N}.$ We define $S(n)$ to be the number of partitions $\lambda$ of $n$ such that
    \[
    \lambda_i-\lambda_{i+1} \geq
    \begin{cases}
        6, &\text{if } \lambda_i \equiv \lambda_{i+1} \equiv 0\;(\text{mod }3), \\
        3, &\text{otherwise.}
    \end{cases}
    \]
\end{definition}
\begin{theorem}[Schur, 1926]\label{thm:Schur}
Let $n \in \mathbb{N}.$ Then, $S(n)$ is equal to the number of partitions of $n$ into distinct parts not divisible by $3.$
\end{theorem}
In terms of generating functions, this corresponds to
    \[
    \sum_{n\geq0} S(n) q^n = (-q;q^3)_\infty(-q^2;q^3)_\infty.
    \]

This theorem was proved in many ways and from it arised a large number of new results. Among them, we can find several proofs by Andrews: two in the late 1960s~\cite{Andrews1967, Andrews1968} and another in the mid-1990s~\cite{Andrews1994}. His many works led to a bijective proof of Schur's theorem by Bessenrodt in 1991~\cite{Bessenrodt1991}. However, in this paper, we are interested in some other bijective proofs, originally based on an extension of Schur's theorem to a general modulus $m\geq3,$ made by Gleissberg in 1928~\cite{Gleissberg1928}:
\begin{definition}
Let $n,k\in\mathbb{N}.$ Let $m,r \in \mathbb{N}_{>0}$ with $r<\frac{m}{2}.$ We define
\begin{enumerate}
\item[(a)] $C_{k,r}(n)$ to be the number of partitions of $n$ into $k$ distinct parts $\equiv \pm r\,(\text{mod }m),$ \item[(b)] $D_{k,r}(n)$ to be the number of partitions $\lambda$ of $n$ into $k$ parts $\equiv 0,\pm r\;(\text{mod }m)$ such that
    \begin{enumerate}
        \item[(i)] the parts $\lambda_i\equiv0\,(\text{mod }m)$ are counted twice,
        \item[(ii)] \[
    \lambda_i-\lambda_{i+1} \geq
    \begin{cases}
        2m, &\text{if } \lambda_i \equiv \lambda_{i+1} \equiv 0\;(\text{mod }m), \\
        m, &\text{otherwise.}
    \end{cases}
    \]
    \end{enumerate}
\end{enumerate}
\end{definition}
\begin{theorem}[Gleissberg, 1928]\label{thm:Gleissberg}
    Let $k,m,r \in \mathbb{N}_{>0}$ with $r<\frac{m}{2}.$ Then,
\[
C_{k,r}(n)=D_{k,r}(n).
\]
\end{theorem}

\begin{remark}
In the case $m=3,$ $r=1$ and summing over all possible values of $k,$ Theorem~\ref{thm:Gleissberg} reduces to Theorem~\ref{thm:Schur}.
\end{remark}

D.~Bressoud gave bijective proofs of both Schur's and Gleissberg's theorems~\cite{Bressoud1980}, as well as bijective proofs for results of Göllnitz~\cite{Bressoud1979}. These contributions in turn motivated Alladi and Gordon to further extend Schur's theorem~\cite{AlladiGordon1993}. To prove their extension, they actually proved another result on partitions, stated below, where the parts are \emph{colored} instead of being under the constraint of some congruences. Colored partitions are defined as follows:
\begin{definition}
Let $n,\,k\in\mathbb{N},$ $k>0.$ Let $c_1,c_2,\dots,c_L \in \{1,\dots,k\}.$ Let $\lambda=(\lambda_1,\lambda_2,\dots,\lambda_L)$ be an integer partition of $n.$ Then,
\[
\lambda=(\,{(\lambda_1)}_{c_1},\,{(\lambda_2)}_{c_2},\,\dots,\,{(\lambda_L)}_{c_L}),
\]
is an integer partition of $n$ in (at most) $k$ colors, where $c_i$ is the color of the part $\lambda_i$.
\end{definition}
The associated generating function is
\begin{equation}\label{eq:gfcoloredpartitions}
   \frac{1}{(y_1q;q)_\infty(y_2q;q)_\infty \cdots (y_kq;q)_\infty}.
\end{equation}
Here the exponent of $y_j$ tracks the number of parts of color $j$, while the 
exponent of $q$ records the integer being partitioned. Since a colored partition 
can be regarded as a \(k\)-tuple of ordinary partitions, one in each color, the 
generating function is simply the product of the $k$ generating functions
\[
   \frac{1}{(y_j q;q)_\infty},
\]
for \(j=1,\dots,k\).

Alladi and Gordon named their idea of using colored partitions instead of congruences the \emph{method of weighted words}.
\begin{definition}\label{def:VetC}
Let $n,i,j\in\mathbb{N}.$ Consider the colors $1,$ $2$ and $3.$
\begin{enumerate}
\item[(a)] We define $V(i,j;n)$ to be the number of ordered pairs of partitions $(\lambda,\mu)$ such that $\lambda$ has $i$ distinct parts in color $1$ and $\mu$ has $j$ distinct parts in color $2$, and where the total sum of the parts in the ordered pair is $n$.

\item[(a)] We define $R(i,j;n)$ to be the number of partitions $\pi$ of $n$ into distinct parts, with $i$ parts in color $1$ or color $3$, and $s$ parts in color $2$ or $3$, such that
    \begin{enumerate}
        \item[(i)] a part of size $1$ cannot be in color $3,$
        \item[(ii)] for all $r$,
        \[
        \pi_r-\pi_{r+1} \geq
        \begin{cases}
            2, &\text{if } \pi_r \text{ is in color $3,$ or if } \pi_r \text{ is in color $1$ and } \pi_{r+1} \text{ is in color $2,$} \\
            1, &\text{otherwise.}
        \end{cases}
        \]
    \end{enumerate}
\end{enumerate}
\end{definition}
\begin{theorem}[Alladi \& Gordon, 1993]\label{thm:AlladiGordonVR}
In the setting of Definition~\ref{def:VetC}, the following identity holds:
\[
    V(i,j;n)=R(i,j;n).
\]
\end{theorem}
The generating function for $V(i,j;n)$ is $(-y_1q;q)_{\infty}(-y_2q;q)_{\infty}.$ Therefore, in terms of generating functions, Theorem~\ref{thm:AlladiGordonVR} states that
\[
\sum_{n,i,j\geq0}R(i,j;n)y_1^iy_2^jq^n=(-y_1q;q)_{\infty}(-y_2q;q)_{\infty}.
\]

Note that under the dilation $q \mapsto q^3$ and translations $y_1\mapsto y_1q^{-2},$ $y_2 \mapsto y_2q^{-1},$ Theorem~\ref{thm:AlladiGordonVR} yields Schur's theorem. Other dilations or translations would lead to new but related results on partitions, which is one of the reasons why the method of weighted words is useful.

Apart from being used for this result, their method also provided a connection between Schur's theorem and multinomial coefficients and was essential in proving that there are six companion results to Schur's theorem~\cite{AlladiGordon1995}. Their method was also used in joint works with E. Andrews to generalize other partition problems, such as Göllnitz's theorems~\cite{AlladiAndrewsGordon1995} or Capparelli's theorem~\cite{AlladiAndrewsGordon1995-2}.

In 1993, Alladi and Gordon gave two proofs of Theorem~\ref{thm:AlladiGordonVR}~\cite{AlladiGordon1993}. The first one uses generating functions and multinomial coefficients. The second one is a bijective proof different from Bressoud's bijective proof of Schur's theorem~\cite{Bressoud1980}, but inspired by ideas of Bressoud on theorems of Göllnitz~\cite{Bressoud1979}. In 2006, Corteel and Lovejoy~\cite{CorteelLovejoy2006} had the idea to iterate this bijective proof in a particular way, in order to have an interpretation of the following infinite product as a generating function for certain partitions where the parts come in $2^k-1$ colors:
\[
(-y_1q;q)_{\infty}(-y_2q;q)_{\infty}\cdots(-y_kq;q)_{\infty}.
\]
To better understand it, we consider the use of \textsl{primary colors}. To be precise, we define $\{2^j\,:\,0 \leq j < k\}$ to be the set of primary colors. Sums of these primary colors produce new colors, and the complete set of colors obtained in this way is $\{1,2,3,\dots,2^k-1\}.$ The integer partitions that will be considered will take their colors in this last set.
\begin{definition}\label{def:notations}
Let $c$ denote the color of a part. Let $c_i$ and $c_{i+1}$ denote the colors of adjacent parts $\lambda_{i}$ and $\lambda_{i+1}.$
\begin{enumerate}
    \item[(a)] We define $\omega(c)$ to be the number of powers of $2$ (\textsl{primary colors}) in the binary representation of $c$:
\[
\omega(c) \coloneqq \sum_{k=0}^{\infty} a_k,
\]
where $c=\sum_{k=0}^{\infty} a_k 2^k$ and $a_k\in\{0,1\}$ for all $k.$
    \item[(b)] We define $v(c) \text{ (resp. $z(c)$)}$ to be the smallest (resp. largest) power of $2$ (\textsl{primary color}) occurring in the binary representation of $c.$
    \item[(c)] We define $\delta(c_i,c_{i+1})$ as follows:
    \[
\delta(c_i,c_{i+1})\coloneqq
    \begin{cases}
        1, &\text{if } z(c_i)<v(c_{i+1}), \\
        0, &\text{otherwise}.
    \end{cases}
\]
\end{enumerate}
\end{definition}
\begin{definition}\label{def:AetB}
Let $x_1,\dots,x_k,m,n \in \mathbb{N}.$ We define
\begin{enumerate}
\item[(a)] $A(x_1,\dots,x_k;n)$ to be the number of $k$-tuples $(\mu_1,\dots,\mu_k),$ where $\mu_r$ is a partition into $x_r$ distinct parts for all $r,$ and the sum of all of the parts in the $k$-tuple is $n,$
\item[(b)] $B(x_1,\dots,x_k;n)$ to be the number of partitions $\lambda=(\lambda_1,\dots,\lambda_L)$ of $n$ into distinct parts and in $2^k-1$ colors, i.e $c_1,\dots,c_L\in\{1,\dots,2^k-1\},$ satisfying the following conditions:
    \begin{enumerate}
      \renewcommand{\labelenumi}{(\roman{enumi})} 
      \item[(i)] the smallest part satisfies $\lambda_L \geq \omega(c_L),$ 
      \item[(ii)] $x_r$ parts have $2^{r-1}$ in their color's binary representation, for all $r\in\{1,\dots,k\},$
      \item[(iii)] $\lambda_i-\lambda_{i+1} \geq \omega(c_i) + \delta(c_i,c_{i+1}),$ for all $i\in\{1,\dots,L-1\}.$
    \end{enumerate}
\end{enumerate}
\end{definition}
\begin{theorem}[Corteel \& Lovejoy, 2006]\label{thm:CorteelLovejoy}
In the setting of Definition~\ref{def:AetB}, the following identity holds:
\[
A(x_1,\dots,x_k;n)=B(x_1,\dots,x_k;n).
\]
\end{theorem}
The generating function for $A(x_1,\dots,x_k;n)$ is $(-y_1 q;q)_\infty \cdots (-y_k q;q)_\infty.$ Therefore, in terms of generating functions, Theorem~\ref{thm:CorteelLovejoy} states that
    \[
    \sum_{\substack{x_1, \dots, x_k, m, n \geq 0}} B(x_1,\dots,x_k;n) y_1^{x_1} \cdots y_k^{x_k} q^n = (-y_1 q;q)_\infty \cdots (-y_k q;q)_\infty.
    \]

Around that time, Corteel and Lovejoy also introduced a new type of integer partitions called \emph{overpartitions}~\cite{CorteelLovejoy2004}.
\begin{definition}
Let $n\in\mathbb{N}.$ An overpartition of $n$ is an integer partition of $n$ where the first occurrence of each value can be overlined.
\end{definition}
The generating function for overpartitions is
\begin{equation}\label{eq:gfoverpartitions}
    \frac{(-q;q)_\infty}{(dq;q)_\infty},
\end{equation}
where the exponent of $d$ counts the non-overlined parts, and the exponent of 
$q$ records the integer being partitioned. This follows from viewing an overpartition 
as a pair of a partition into distinct parts (overlined) and an unrestricted partition 
(non-overlined), so the generating function factors accordingly.

\begin{example}
Consider the overpartition $(4,4,\overline{3},3,\overline{1})$ of $15.$ It can be represented graphically using a Young diagram:
    \[
    \ytableausetup{centertableaux}
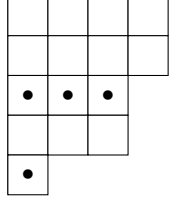
    \begin{YD}
        \boxrow{4/white/}
        \boxrow{4/white/}
        \boxrow{3/white/\bullet}
        \boxrow{3/white/}
        \boxrow{1/white/\bullet}
    \end{YD}
    \]
where the rows with dots in them represent the overlined parts (in this case $\overline{3}$ and $\overline{1}$).
\end{example}

In 2005, Lovejoy considered colored overpartitions to generalize Schur's theorem for overpartitions~\cite{Lovejoy2005}.
\begin{definition}\label{def:Doverlined}
Let $x_1,x_2,m,n\in\mathbb{N}.$ We define $\overline{D}(x_1,x_2;m,n)$ to be the number of overpartitions $\lambda=(\lambda_1,\dots,\lambda_L)$ of $n$ in $3$ colors, i.e $c_1,\dots,c_L\in\{1,2,3\},$ having $m$ non-overlined parts and satisfying the following conditions:
\begin{enumerate}
\renewcommand{\labelenumi}{(\roman{enumi})} 
\item no part $1_{3}$ or $\overline{1_{3}},$
\item the number of parts with color $1$ or $3$ is $x_1$ and the number of parts with color $2$ or $3$ is $x_2,$
\item the entry $(a,b)$ of the following matrix gives the minimal difference between $\lambda_i$ of color $a$ and $\lambda_{i+1}$ of color $b$ (where the label $d$ refers to non-overlined parts):
\[
\begin{blockarray}{ccccccc}
& 1 & 2 & 3 & 1d & 2d & 3d \\
\begin{block}{c(cccccc)}
1 & 1 & 2 & 1 & 0 & 1 & 0 \\
2 & 1 & 1 & 1 & 0 & 0 & 0 \\
3 & 2 & 2 & 2 & 1 & 1 & 1 \\
1d & 1 & 2 & 1 & 0 & 1 & 0 \\
2d & 1 & 1 & 1 & 0 & 0 & 0 \\
3d & 2 & 2 & 2 & 1 & 1 & 1 \\
\end{block}
\end{blockarray}.
 \]
\end{enumerate}
\end{definition}
\begin{remark}
Note that the third condition is equivalent to the following:
\[
(\widetilde{iii})\ \lambda_i-\lambda_{i+1} \geq \omega(c_i) + \delta(c_i,c_{i+1}) - \mathbf{1}_{\lambda_{i+1} \text{ is not overlined}},\text{ for all }i\in\{1,\dots,L-1\},
\]
where $c_1,\dots,c_L\in\{1,2,3\}$ denote the colors of the parts $\lambda_1,\dots,\lambda_L.$
This can be verified by straightforward calculations.
\end{remark}
\begin{theorem}[Lovejoy, 2005]\label{thm:SchurOverpartitions}
In the setting of Definition~\ref{def:Doverlined}, the following identity holds:
\[
\sum_{\substack{x_1, x_2, m, n \geq 0}} \overline{D}(x_1,x_2;m,n) y_1^{x_1} y_2^{x_2} d^m q^n = \frac{(-y_1 q;q)_\infty (-y_2 q;q)_\infty}{(y_1 d q;q)_\infty(y_2 d q;q)_\infty}.
\]
\end{theorem}

Around a decade later, with the method of weighted words still in mind, Lovejoy~\cite{Lovejoy2017} gave an asymmetrical version of Schur's theorem for overpartitions and proved it in two ways: using generating functions and facts about $q$-series, and bijectively with a similar proof to the one given for Theorem~\ref{thm:AlladiGordonVR} in~\cite{AlladiGordon1993}, which was inspired by Bressoud's.
\begin{definition}\label{def:D1etD2}
Let $x_1,x_2,m,n\in\mathbb{N}.$ We define
\begin{enumerate}
\item[(a)] $\overline{D_1}(x_1,x_2;m,n)$ to be the same as $\overline{D}(x_1,x_2;m,n)$ from Definition~\ref{def:Doverlined} to which we add a fourth condition:
\[
(iv) \text{ the }s\text{ smallest parts must be overlined, where }s\text{ is the number of parts of color }2.
\]
\item[(b)] $\overline{D_2}(x_1,x_2;m,n)$ to be the same as $\overline{D}(x_1,x_2;m,n)$ from Definition~\ref{def:Doverlined} to which we add a fourth condition:
\[
(\widetilde{iv}) \text{ the }r\text{ smallest parts must be overlined, where }r\text{ is the number of parts of color }1.
\]
\end{enumerate}
\end{definition}
\begin{theorem}[Lovejoy, 2017]\label{thm:Lovejoy2017}
In the setting of Definition~\ref{def:D1etD2}, the following identities hold:
\begin{enumerate}
\item[(a)] \begin{equation}\label{eq:Soverlined}
\sum_{\substack{x_1, x_2, m, n \geq 0}} \overline{D_1}(x_1,x_2;m,n) y_1^{x_1} y_2^{x_2} d^m q^n = \frac{(-y_1 q;q)_\infty (-y_2 q;q)_\infty}{(y_1 d q;q)_\infty}.
\end{equation}
\item[(b)] \begin{equation}\label{eq:Roverlined}
\sum_{\substack{x_1, x_2, m, n \geq 0}} \overline{D_2}(x_1,x_2;m,n) y_1^{x_1} y_2^{x_2} d^m q^n = \frac{(-y_1 q;q)_\infty (-y_2 q;q)_\infty}{(y_2 d q;q)_\infty}.
\end{equation}
\end{enumerate}
\end{theorem}

In this paper, we extend the first identity of Lovejoy’s asymmetric theorem (equation~\eqref{eq:Soverlined}) to a broader class of infinite products. More specifically, we provide a combinatorial interpretation of the following infinite product as the generating function for certain overpartitions, in which parts are colored using $2^k-1$ colors:
\[
\frac{(-y_1 q;q)_\infty \cdots (-y_k q;q)_\infty}{(y_1 d q;q)_\infty}.
\]
Our method is to iterate the bijective proof of Theorem~\ref{thm:Lovejoy2017} by Lovejoy~\cite{Lovejoy2017}. This strategy is inspired by the approach of Corteel and Lovejoy in deriving Theorem~\ref{thm:CorteelLovejoy} from Theorem~\ref{thm:AlladiGordonVR}~\cite{CorteelLovejoy2006}. To state our result precisely, we first give a definition of a particular class of overpartitions.
\begin{definition}\label{def:Soverlined}
Let $x_1,\dots,x_k,m,n \in \mathbb{N}.$ We define $\overline{S}(x_1,\dots,x_k;m,n)$ to be the number of overpartitions $\lambda=(\lambda_1,\dots,\lambda_L)$ of $n$ in $2^k-1$ colors, i.e $c_1,\dots,c_L\in\{1,\dots,2^k-1\},$ having $m$ non-overlined parts and satisfying the following conditions:
\begin{enumerate}
  \renewcommand{\labelenumi}{(\roman{enumi})} 
  \item the smallest part satisfies $\lambda_L \geq \omega(c_L),$ 
  \item $x_i$ parts have $2^{i-1}$ in their color's binary decomposition, for $i=1,\dots,k,$
  \item $\lambda_i-\lambda_{i+1} \geq \omega(c_i) + \delta(c_i,c_{i+1}) - \mathbf{1}_{\lambda_{i+1} \text{ is not overlined}},$ for $i=1,\dots,L-1,$
  \item the $s$ smallest parts are overlined, where $s \coloneqq \sum_{r=1}^{k-1} V_{2^r}$ and $V_{2^r}$ is the number of parts $\lambda_i$ such that $v(c_i)=2^r.$
\end{enumerate}
We denote $\mathfrak{S}(x_1,\dots,x_k;m,n)$ the set of all such overpartitions.
\end{definition}

Our main result can be stated as follows.

\begin{theorem}\label{thm:LV}
Let $x_1,\dots,x_k,m,n \in \mathbb{N}.$ Then,
\begin{equation}\label{eq:thm1}
\sum_{\substack{x_1, \dots, x_k, m, n \geq 0}} \overline{S}(x_1,\dots,x_k;m,n) y_1^{x_1} \cdots y_k^{x_k} d^m q^n = \frac{(-y_1 q;q)_\infty \cdots (-y_k q;q)_\infty}{(y_1 d q;q)_\infty}.
\end{equation}
\end{theorem}

The proof of Theorem~\ref{thm:LV} is a bijection by induction on $k.$ As stated earlier, our approach unifies Lovejoy’s bijective proof of Theorem~\ref{thm:Lovejoy2017}~\cite{Lovejoy2017} and the iterative-bijective technique developed by Corteel and Lovejoy for the proof of Theorem~\ref{thm:CorteelLovejoy}~\cite{CorteelLovejoy2006}. Indeed, Theorem~\ref{thm:LV} reduces to
Theorem~\ref{thm:CorteelLovejoy} when $d=0$, i.e when there are no overlined parts, and it reduces to Theorem~\ref{thm:Lovejoy2017} when $k=2$.

The bijective proof of Theorem~\ref{thm:Lovejoy2017}, adapted to our notations, is reviewed in Section~\ref{subsec:caseKis2} and presented alongside an illustrative example using Young diagrams. This preparation should help the reader follow the general case -- the induction step -- proved in Section~\ref{subsec:caseKbigger3}. In that last section, we also provide an example for the general case, though without Young diagrams, as the values involved are too large.

\section{Proof of Theorem~\ref{thm:LV}}\label{sec:proofnewthm}

\subsection{Case $k=1$}\label{subsec:caseKis1}

In this section, we prove the base case $k=1$ for the inductive proof. In this setting, the statement becomes
    \[
    \sum_{\substack{x_1, m, n \geq 0}} \overline{S}(x_1;m,n) y_1^{x_1} d^m q^n = \frac{(-y_1 q;q)_\infty}{(y_1 d q;q)_\infty},
    \]
    where $\overline{S}(x_1;m,n)$ denotes the number of overpartitions $\lambda = (\lambda_1,\dots,\lambda_L)$ of $n$ in $1$ color, having $m$ non-overlined parts and satisfying the following conditions:
    \begin{enumerate}
      \renewcommand{\labelenumi}{(\roman{enumi})} 
      \item the smallest part satisfies $\lambda_L \geq 1,$ 
      \item $x_1$ parts have $1$ in their color's binary decomposition,
      \item $\lambda_i-\lambda_{i+1} \geq 1 - \mathbf{1}_{\lambda_{i+1} \text{ is not overlined}},$ for $i=1,\dots,L-1,$
      \item the $s$ smallest parts are overlined, where $s=0.$
    \end{enumerate}
    In this case, $\overline{S}(x_1;m,n)$ is exactly the number of overpartitions of $n$ in $1$ color, having $x_1$ parts and $m$ non-overlined parts. Indeed, the four conditions are always true for any overpartitions. This is clear for (i), (ii), and (iv). For (iii), we have two cases:
    \begin{enumerate}
    \item If $\lambda_{i+1}$ is overlined, then $\lambda_i > \lambda_{i+1}$ by definition of an overpartition, and $\lambda_i-\lambda_{i+1} \geq 1=1 - \mathbf{1}_{\lambda_{i+1} \text{ is not overlined}}.$
    \item If $\lambda_{i+1}$ is not overlined, then $\lambda_i \geq \lambda_{i+1}$ by definition of an overpartition, and $\lambda_i-\lambda_{i+1} \geq 0=1 - \mathbf{1}_{\lambda_{i+1} \text{ is not overlined}}.$
    \end{enumerate}
    Multiplying $\overline{S}(x_1;m,n)$ by $y_1^{x_1} d^m q^n$ and summing over $x_1, m, n,$ we get the generating function for overpartitions in 1 color (see \eqref{eq:gfoverpartitions}), which is exactly the right-hand side of our statement in the case $k=1.$

\subsection{Case $k=2$}\label{subsec:caseKis2}

When $k=2,$ Theorem~\ref{thm:LV} reduces to the first part of Theorem~\ref{thm:Lovejoy2017}. Since our generalization is partly based on this result, and to help the reader navigate the numerous notations, we present here an adaptation of Lovejoy’s proof~\cite{Lovejoy2017} to our notations, enhanced with additional details and an illustrative example for clarity.

\medskip

Let $x_1,x_2,m,n\in\mathbb{N}.$ Consider $\lambda$ an overpartition with $x_1$ parts all in color $1$ and $\mu$ a partition into $x_2$ distinct parts all in color $2.$
\begin{example}
    
        We will use the colors $1=$ \textcolor{Goldenrod}{yellow}, $2=$ \textcolor{Aquamarine}{blue}, $3=$ \textcolor{LimeGreen}{green}.
        
        \smallskip
        
        \noindent Let $\lambda=(\overline{8}_1,\overline{6}_1,6_1,\overline{4}_1,3_1,1_1)$ be an overpartition in color $1$ and $\mu=(8_2,7_2,3_2,1_2)$ a partition into distinct parts in color $2$:
        \[
        \ytableausetup{centertableaux}
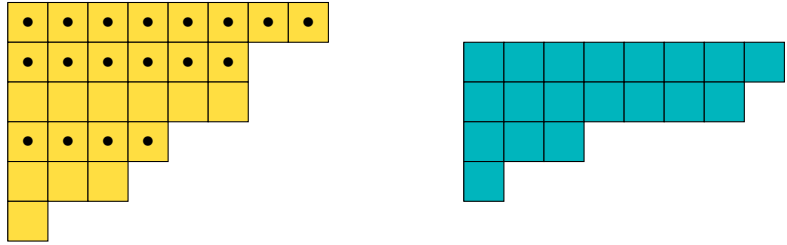
        \begin{tikzpicture}[baseline=(current bounding box.center)]
        \node (first) at (0,0) {
        \begin{YD}
            \boxrow{8/Goldenrod/\bullet}
            \boxrow{6/Goldenrod/\bullet}
            \boxrow{6/Goldenrod/}
            \boxrow{4/Goldenrod/\bullet}
            \boxrow{3/Goldenrod/}
            \boxrow{1/Goldenrod/}
        \end{YD}
        };
    
        \node (second) at (6,0) {
        \begin{YD}
            \boxrow{8/Aquamarine/}
            \boxrow{7/Aquamarine/}
            \boxrow{3/Aquamarine/}
            \boxrow{1/Aquamarine/}
        \end{YD}
        };
        \end{tikzpicture}
        \]
\end{example}

\begin{enumerate}[label=\textbf{Step \arabic*:}]
    \item Let $L$ denote the number of parts in $\lambda.$ For each part $p$ of $\mu$ such that $p\leq L,$ add $1$ to the first $p$ parts of $\lambda$ and change the color of $\lambda_p$ to $3$ if the original color was $1.$ This gives a new overpartition $\lambda^{(1)}.$ Keeping only the unused parts of $\mu,$ we also get a new partition into distinct parts, denoted $\mu^{(1)}.$
\begin{example}
    
        Using the partitions defined in the previous example, we apply Step 1. First, we find the parts in $\mu$ that are smaller or equal to the number of parts in $\lambda,$ i.e $p\leq6$:
        \[
        \ytableausetup{centertableaux}
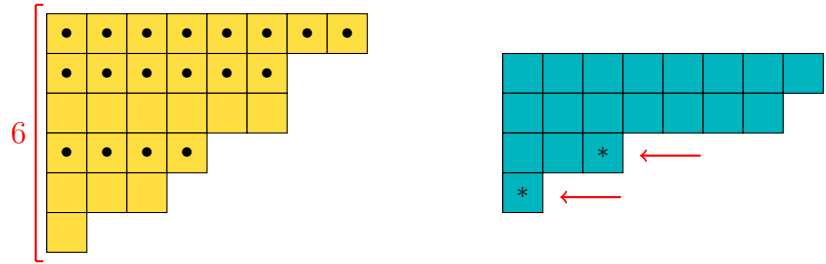
        \begin{tikzpicture}[baseline=(current bounding box.center)]
        \node (first) at (0,0) {
        \begin{YD}
            \boxrow{8/Goldenrod/\bullet}
            \boxrow{6/Goldenrod/\bullet}
            \boxrow{6/Goldenrod/}
            \boxrow{4/Goldenrod/\bullet}
            \boxrow{3/Goldenrod/}
            \boxrow{1/Goldenrod/}
        \end{YD}
        };
    
        \node (second) at (6,0) {
        \begin{YD}
            \boxrow{8/Aquamarine/}
            \boxrow{7/Aquamarine/}
            \boxrow{2/Aquamarine/,1/Aquamarine/\ast}
            \boxrow{1/Aquamarine/\ast}
        \end{YD}
        };

        \draw[->, thick, red] (6.55,-0.3) -- (5.75,-0.3); 
        \draw[->, thick, red] (5.5,-0.85) -- (4.7,-0.85); 

        \draw[thick,red] (-2.2, 1.7) -- (-2.2, -1.7) node[midway, left] {\large 6};
        \draw[thick,red] (-2.2, 1.7) -- (-2.1, 1.7);
        \draw[thick,red] (-2.2, -1.7) -- (-2.1, -1.7);
    
        \end{tikzpicture}
        \]
        We transpose those blue parts smaller than $6$:
        \[
        \ytableausetup{centertableaux}
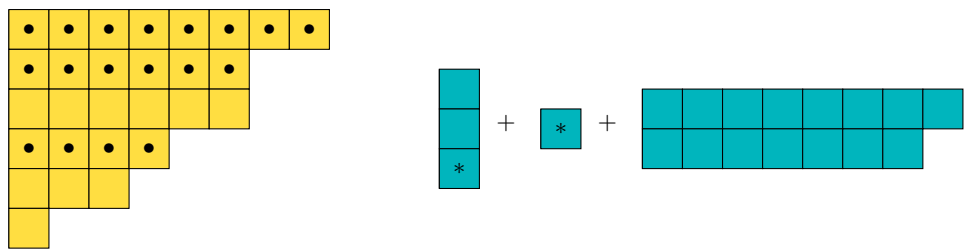
        \begin{tikzpicture}[baseline=(current bounding box.center)]
        \node (first) at (0,0) {
        \begin{YD}
            \boxrow{8/Goldenrod/\bullet}
            \boxrow{6/Goldenrod/\bullet}
            \boxrow{6/Goldenrod/}
            \boxrow{4/Goldenrod/\bullet}
            \boxrow{3/Goldenrod/}
            \boxrow{1/Goldenrod/}
        \end{YD}
        };

        \node (second) at (7,0) {
        \begin{YD}
            \boxrow{1/Aquamarine/}
            \boxrow{1/Aquamarine/}
            \boxrow{1/Aquamarine/\ast}
        \end{YD}
        +
        \begin{YD}
            \boxrow{1/Aquamarine/\ast}
        \end{YD}
        +
        \begin{YD}
            \boxrow{8/Aquamarine/}
            \boxrow{7/Aquamarine/}
        \end{YD}
        };
        \end{tikzpicture}
        \]
        We add them to the top right of $\lambda$:
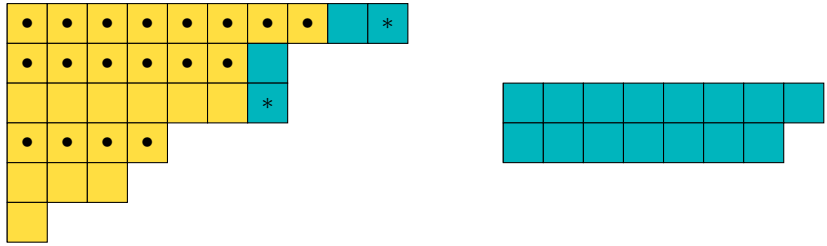
        \[
        \ytableausetup{centertableaux}
        \begin{tikzpicture}[baseline=(current bounding box.center)]
        \node (first) at (0,0) {
        \begin{YD}
            \boxrow{8/Goldenrod/\bullet,1/Aquamarine/,1/Aquamarine/\ast}        \boxrow{6/Goldenrod/\bullet,1/Aquamarine/}            \boxrow{6/Goldenrod/,1/Aquamarine/\ast}
            \boxrow{4/Goldenrod/\bullet}
            \boxrow{3/Goldenrod/}
            \boxrow{1/Goldenrod/}
        \end{YD}
        };
    
        \node (second) at (6,0) {
        \begin{YD}
            \boxrow{8/Aquamarine/}
            \boxrow{7/Aquamarine/}
        \end{YD}
        };
        \end{tikzpicture}
        \]
        We extend the yellow lines and change the color of the lines ending with $\ast$ to the color green:
        \[
        \ytableausetup{centertableaux}
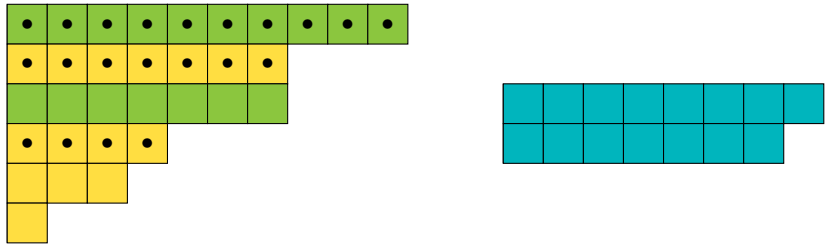
        \begin{tikzpicture}[baseline=(current bounding box.center)]
        \node (first) at (0,0) {
        \begin{YD}
            \boxrow{10/LimeGreen/\bullet}
            \boxrow{7/Goldenrod/\bullet}            \boxrow{7/LimeGreen/}
            \boxrow{4/Goldenrod/\bullet}
            \boxrow{3/Goldenrod/}
            \boxrow{1/Goldenrod/}
        \end{YD}
        };
    
        \node (second) at (6,0) {
        \begin{YD}
            \boxrow{8/Aquamarine/}
            \boxrow{7/Aquamarine/}
        \end{YD}
        };
        \end{tikzpicture}
        \]
        We have obtained:
        \[
        \lambda^{(1)}=(\overline{10}_3,\overline{7}_1,7_3,\overline{4}_1,3_1,1_1) \quad\text{and}\quad \mu^{(1)}=(8_2,7_2).
        \]
\end{example}

\item Let $M$ denote the number of parts in the partition $\mu^{(1)}.$ From $\mu^{(1)},$ get $\mu^{(2)}$ by taking away a (colorless) staircase of size $(\,L+M-1,\,L+M-2,\dots,\,L\,).$ This is possible, since the smallest part in $\mu^{(1)}$ is bigger than $L$ and all its parts are distinct. From $\lambda^{(1)},$ get a partition $\lambda^{(2)}$ that isn't overlined anymore by taking away a (colorless) \emph{generalized staircase}, i.e for each overlined part in $\lambda^{(1)},$ remove the overline and remove $1$ from every part above this one. These $1$'s are moved to the generalized staircase as a new part.
\begin{remark}
A generalized staircase can have a part of size $0$ (e.g. if the largest part in $\lambda^{(1)}$ is overlined).
\end{remark}
Since the largest part of the generalized staircase is at most $L-1,$ combine it with the staircase $(\,L+M-1,\,L+M-2,\dots,\,L\,).$ This gives a \emph{partial staircase}, which will be denoted $\nu.$
\begin{example}
    
        We apply Step 2 to the partitions obtained after Step 1 in the previous example:
        \[
        \ytableausetup{centertableaux}
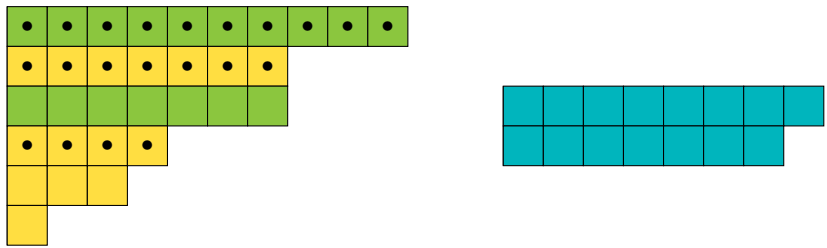
        \begin{tikzpicture}[baseline=(current bounding box.center)]
        \node (first) at (0,0) {
        \begin{YD}
            \boxrow{10/LimeGreen/\bullet}
            \boxrow{7/Goldenrod/\bullet}            \boxrow{7/LimeGreen/}
            \boxrow{4/Goldenrod/\bullet}
            \boxrow{3/Goldenrod/}
            \boxrow{1/Goldenrod/}
        \end{YD}
        };
    
        \node (second) at (6,0) {
        \begin{YD}
            \boxrow{8/Aquamarine/}
            \boxrow{7/Aquamarine/}
        \end{YD}
        };
        \end{tikzpicture}
        \]
        First, we extract step by step a generalized staircase from $\lambda^{(1)}$:
        \begin{itemize}
            \item[--] The first part is overlined. Since there are no parts above this one, we put a part of size $0$ in the generalized staircase. It is represented by a vertical line. The overline is also removed from the first part:
            \[
            \ytableausetup{centertableaux}
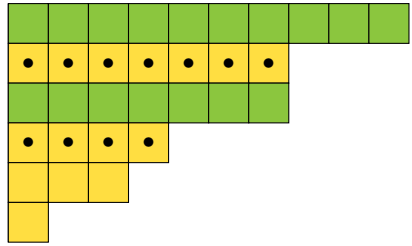
            \begin{tikzpicture}[baseline=(current bounding box.center)]
            \node (first) at (0,0) {
            \begin{YD}
                \boxrow{10/LimeGreen/}
                \boxrow{7/Goldenrod/\bullet}            \boxrow{7/LimeGreen/}
                \boxrow{4/Goldenrod/\bullet}
                \boxrow{3/Goldenrod/}
                \boxrow{1/Goldenrod/}
            \end{YD}
            };
        
            \node (second) at (6,0) {
            \begin{YD}
                \boxrow{}
            \end{YD}
            };
            \end{tikzpicture}
            \]
            
            \item[--] The second part is overlined, which means we extract $1$ from the first part and move it to the generalized staircase:
            \[
            \ytableausetup{centertableaux}
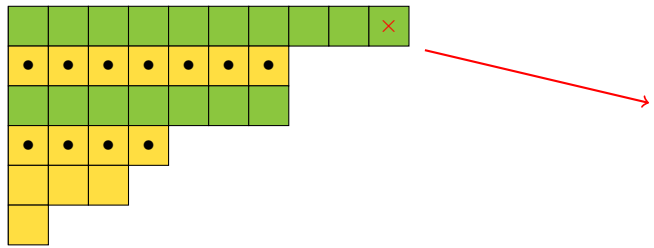
            \begin{tikzpicture}[baseline=(current bounding box.center)]
            \node (first) at (0,0) {
            \begin{YD}
                \boxrow{9/LimeGreen/,1/LimeGreen/\textcolor{red}{\times}}
                \boxrow{7/Goldenrod/\bullet}            \boxrow{7/LimeGreen/}
                \boxrow{4/Goldenrod/\bullet}
                \boxrow{3/Goldenrod/}
                \boxrow{1/Goldenrod/}
            \end{YD}
            };
        
            \node (second) at (6,0) {
            \begin{YD}
                \boxrow{}
            \end{YD}
            };
            
            \draw[->, thick, red] (2.9,1) -- (5.85,0.3); 
            
            \end{tikzpicture}
            \]
            And the overline is removed from the second part:
            \[
            \ytableausetup{centertableaux}
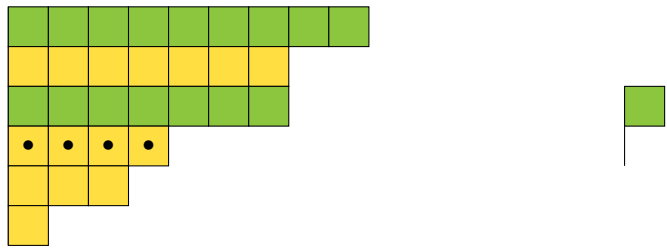
            \begin{tikzpicture}[baseline=(current bounding box.center)]
            \node (first) at (0,0) {
            \begin{YD}
                \boxrow{9/LimeGreen/}
                \boxrow{7/Goldenrod/}            \boxrow{7/LimeGreen/}
                \boxrow{4/Goldenrod/\bullet}
                \boxrow{3/Goldenrod/}
                \boxrow{1/Goldenrod/}
            \end{YD}
            };
        
            \node (second) at (6,0) {
            \begin{YD}
                \boxrow{1/LimeGreen/}
                \boxrow{}
            \end{YD}
            };
            \end{tikzpicture}
            \]
            
            \item[--] The fourth part is overlined, which means we extract $1$ from each of the first three parts and move them to the generalized staircase, as a new part:
            \[
            \ytableausetup{centertableaux}
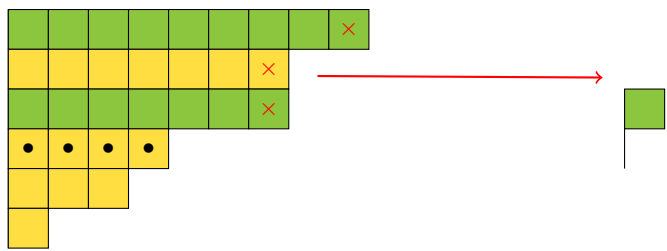
            \begin{tikzpicture}[baseline=(current bounding box.center)]
            \node (first) at (0,0) {
            \begin{YD}
                \boxrow{8/LimeGreen/,1/LimeGreen/\textcolor{red}{\times}}
                \boxrow{6/Goldenrod/,1/Goldenrod/\textcolor{red}{\times}}         \boxrow{6/LimeGreen/,1/LimeGreen/\textcolor{red}{\times}}
                \boxrow{4/Goldenrod/\bullet}
                \boxrow{3/Goldenrod/}
                \boxrow{1/Goldenrod/}
            \end{YD}
            };
        
            \node (second) at (6,0) {
           \begin{YD}
                \boxrow{1/LimeGreen/}
                \boxrow{}
            \end{YD}
            };
            
            \draw[->, thick, red] (1.75,0.7) -- (5.5,0.68); 
            
            \end{tikzpicture}
            \]
            And the overline is removed from the fourth part:
            \[
            \ytableausetup{centertableaux}
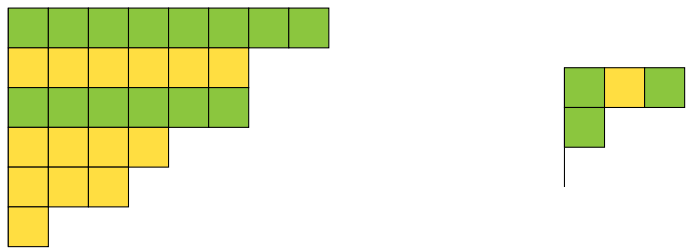
            \begin{tikzpicture}[baseline=(current bounding box.center)]
            \node (first) at (0,0) {
            \begin{YD}
                \boxrow{8/LimeGreen/}
                \boxrow{6/Goldenrod/}            \boxrow{6/LimeGreen/}
                \boxrow{4/Goldenrod/}
                \boxrow{3/Goldenrod/}
                \boxrow{1/Goldenrod/}
            \end{YD}
            };
        
            \node (second) at (6,0) {
            \begin{YD}
                \boxrow{1/LimeGreen/,1/Goldenrod/,1/LimeGreen/}
                \boxrow{1/LimeGreen/}
                \boxrow{}
            \end{YD}
            };
            \end{tikzpicture}
            \]
            
            \item[--] There is actually no color in the generalized staircase. The colors were put only for better understanding of the procedure. Therefore, we can remove the colors:
            \[
            \ytableausetup{centertableaux}
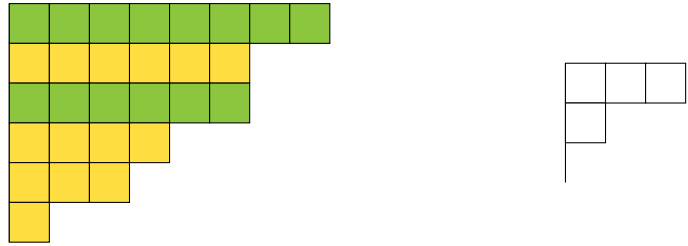
            \begin{tikzpicture}[baseline=(current bounding box.center)]
            \node (first) at (0,0) {
            \begin{YD}
                \boxrow{8/LimeGreen/}
                \boxrow{6/Goldenrod/}            \boxrow{6/LimeGreen/}
                \boxrow{4/Goldenrod/}
                \boxrow{3/Goldenrod/}
                \boxrow{1/Goldenrod/}
            \end{YD}
            };
        
            \node (second) at (6,0) {
            \begin{YD}
                \boxrow{3/white/}
                \boxrow{1/white/}
                \boxrow{}
            \end{YD}
            };
            \end{tikzpicture}
            \]
        \end{itemize}
        Now, we extract from $\mu^{(1)}$ a staircase $(7,6)$:
        \[
        \ytableausetup{centertableaux}
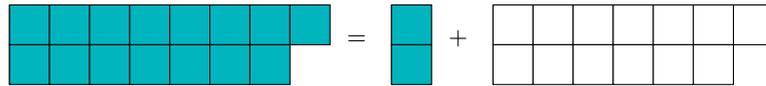
        \begin{tikzpicture}[baseline=(current bounding box.center)]
        \node (first) at (0,0) {
        \begin{YD}
            \boxrow{8/Aquamarine/}
            \boxrow{7/Aquamarine/}
        \end{YD}
        =
        \begin{YD}
            \boxrow{1/Aquamarine/}
            \boxrow{1/Aquamarine/}
        \end{YD}
        +
        \begin{YD}
            \boxrow{7/white/}
            \boxrow{6/white/}
        \end{YD}
        };
        \end{tikzpicture}
        \]
        
        Combining the generalized staircase and the staircase $(7,6)$ in one partial staircase $\nu,$ we get $\lambda^{(2)}=(8_3,6_1,6_3,4_1,3_1,1_1),\,\mu^{(2)}=(1_2,1_2),\,\nu=(7,6,3,1,0)$:
        \[
        \ytableausetup{centertableaux}
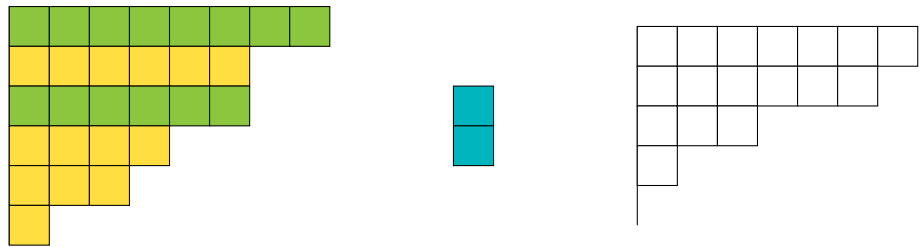
        \begin{tikzpicture}[baseline=(current bounding box.center)]
        \node (first) at (0,0) {
            \begin{YD}
                \boxrow{8/LimeGreen/}
                \boxrow{6/Goldenrod/}            \boxrow{6/LimeGreen/}
                \boxrow{4/Goldenrod/}
                \boxrow{3/Goldenrod/}
                \boxrow{1/Goldenrod/}
            \end{YD}
        };
    
        \node (second) at (4,0) {
            \begin{YD}
                \boxrow{1/Aquamarine/}
                \boxrow{1/Aquamarine/}
            \end{YD}
        };
    
        \node (third) at (8,0) {
            \begin{YD}
                \boxrow{7/white/}
                \boxrow{6/white/}
                \boxrow{3/white/}
                \boxrow{1/white/}
                \boxrow{}
            \end{YD}
        };
        
        \end{tikzpicture}
        \]
    \end{example}
    
    \item Insert into $\lambda^{(2)}$ each part of $\mu^{(2)}$, starting from the largest, so that it becomes the first occurrence of its value in $\lambda^{(2)}.$ The partition obtained after completing all insertions is denoted $\lambda^{(3)}.$

\begin{example}
    
        We apply Step 3 to the partitions obtained in the previous example:
        \[
        \ytableausetup{centertableaux}
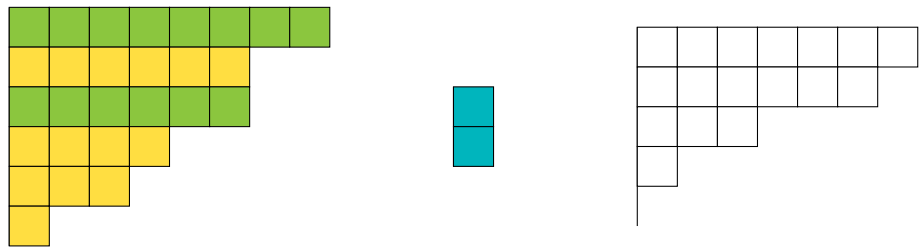
        \begin{tikzpicture}[baseline=(current bounding box.center)]
        \node (first) at (0,0) {
            \begin{YD}
                \boxrow{8/LimeGreen/}
                \boxrow{6/Goldenrod/}            \boxrow{6/LimeGreen/}
                \boxrow{4/Goldenrod/}
                \boxrow{3/Goldenrod/}
                \boxrow{1/Goldenrod/}
            \end{YD}
        };
    
        \node (second) at (4,0) {
            \begin{YD}
                \boxrow{1/Aquamarine/}
                \boxrow{1/Aquamarine/}
            \end{YD}
        };
    
        \node (third) at (8,0) {
            \begin{YD}
                \boxrow{7/white/}
                \boxrow{6/white/}
                \boxrow{3/white/}
                \boxrow{1/white/}
                \boxrow{}
            \end{YD}
        };
        
        \end{tikzpicture}
        \]
        
        We insert all parts of $\mu^{(2)}$ in $\lambda^{(2)},$ so that they become the first occurence occurence of their value in $\lambda^{(2)}$:
        \[
        \ytableausetup{centertableaux}
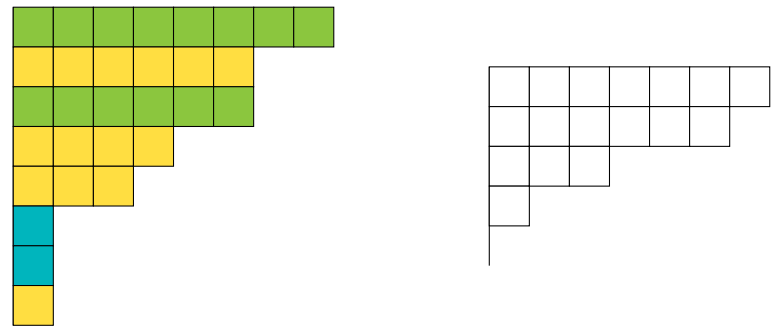
        \begin{tikzpicture}[baseline=(current bounding box.center)]
        \node (first) at (0,0) {
            \begin{YD}
                \boxrow{8/LimeGreen/}
                \boxrow{6/Goldenrod/}            \boxrow{6/LimeGreen/}
                \boxrow{4/Goldenrod/}
                \boxrow{3/Goldenrod/}
                \boxrow{1/Aquamarine/}
                \boxrow{1/Aquamarine/}
                \boxrow{1/Goldenrod/}
            \end{YD}
        };
    
        \node (second) at (6,0) {
            \begin{YD}
                \boxrow{7/white/}
                \boxrow{6/white/}
                \boxrow{3/white/}
                \boxrow{1/white/}
                \boxrow{}
            \end{YD}
        };
        
        \end{tikzpicture}
        \]
We get $\lambda^{(3)}=(8_3,6_1,6_3,4_1,3_1,1_2,1_2,1_1).$
    \end{example}
    
    \item Add the partial staircase $\nu$ back to $\lambda^{(3)}.$ To do so, for each part $p$ of $\nu,$ add $1$ to the first $p$ parts of $\lambda^{(3)}$ and overline the ($p+1$)-th part. This gives $\lambda^{(4)}.$
\begin{example}
    
        We apply Step 4 to the partitions obtained previously:
            \[
        \ytableausetup{centertableaux}
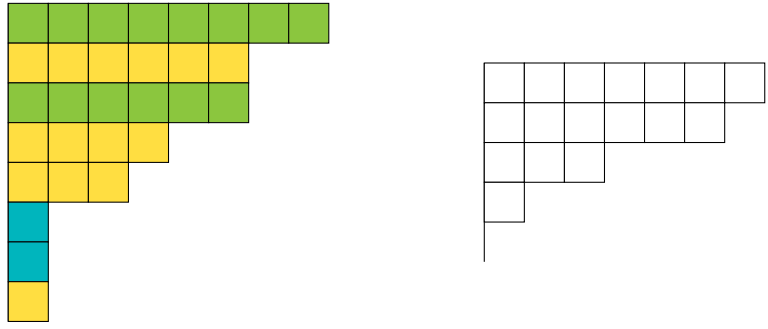
        \begin{tikzpicture}[baseline=(current bounding box.center)]
        \node (first) at (0,0) {
            \begin{YD}
                \boxrow{8/LimeGreen/}
                \boxrow{6/Goldenrod/}            \boxrow{6/LimeGreen/}
                \boxrow{4/Goldenrod/}
                \boxrow{3/Goldenrod/}
                \boxrow{1/Aquamarine/}
                \boxrow{1/Aquamarine/}
                \boxrow{1/Goldenrod/}
            \end{YD}
        };
    
        \node (second) at (6,0) {
            \begin{YD}
                \boxrow{7/white/}
                \boxrow{6/white/}
                \boxrow{3/white/}
                \boxrow{1/white/}
                \boxrow{}
            \end{YD}
        };
        
        \end{tikzpicture}
        \]
       We reinsert the partial staircase $\nu$ in $\lambda^{(3)}.$ To do so, we write $\ast$ in the last box of each row in $\nu$:
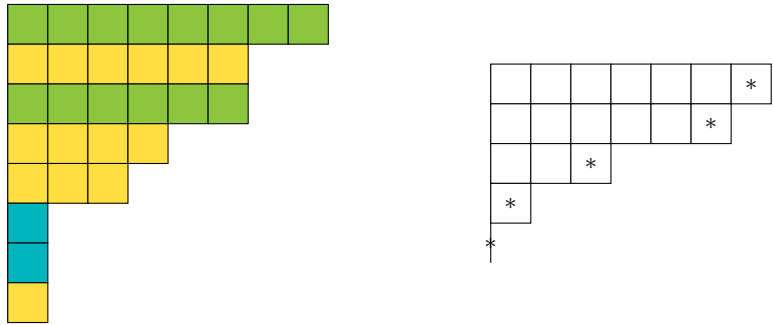
       \[
        \ytableausetup{centertableaux}
        \begin{tikzpicture}[baseline=(current bounding box.center)]
        \node (first) at (0,0) {
            \begin{YD}
                \boxrow{8/LimeGreen/}
                \boxrow{6/Goldenrod/}            \boxrow{6/LimeGreen/}
                \boxrow{4/Goldenrod/}
                \boxrow{3/Goldenrod/}
                \boxrow{1/Aquamarine/}
                \boxrow{1/Aquamarine/}
                \boxrow{1/Goldenrod/}
            \end{YD}
        };
    
        \node (second) at (6,0) {
            \begin{YD}                \boxrow{6/white/,1/white/\ast}                \boxrow{5/white/,1/white/\ast}                \boxrow{2/white/,1/white/\ast}
                \boxrow{1/white/\ast}
                \boxrowa{}
            \end{YD}
        };
        
        \end{tikzpicture}
        \]

        We transpose $\nu$:
        \[
        \ytableausetup{centertableaux}
        
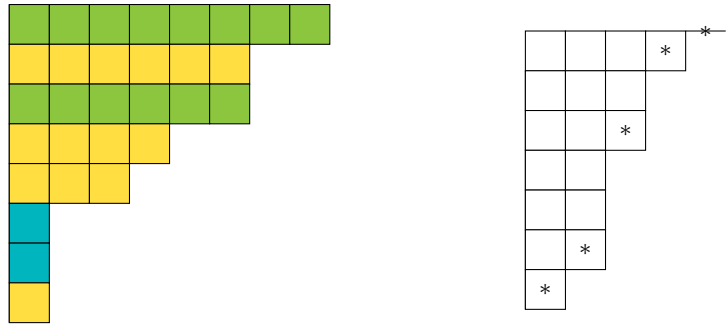
\begin{tikzpicture}[baseline=(current bounding box.center)]
        \node (first) at (0,0) {
            \begin{YD}
                \boxrow{8/LimeGreen/}
                \boxrow{6/Goldenrod/}            \boxrow{6/LimeGreen/}
                \boxrow{4/Goldenrod/}
                \boxrow{3/Goldenrod/}
                \boxrow{1/Aquamarine/}
                \boxrow{1/Aquamarine/}
                \boxrow{1/Goldenrod/}
            \end{YD}
        };
    
        \node (second) at (6,0) {
            \begin{YD}                \boxrowta{3/white/,1/white/\ast}
            \boxrow{3/white/}
            \boxrow{2/white/,1/white/\ast}
            \boxrow{2/white/}
            \boxrow{2/white/}
            \boxrow{1/white/,1/white/\ast}
            \boxrow{1/white/\ast}          
            \end{YD}
        };
        
        \end{tikzpicture}
        \]
        We add it to the top right of $\lambda^{(3)}$:
        \[
        \ytableausetup{centertableaux}
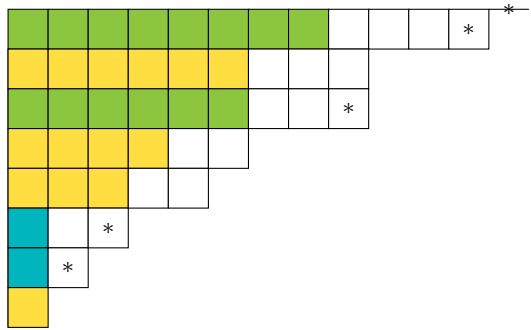
        \begin{tikzpicture}[baseline=(current bounding box.center)]
        \node (first) {
            \begin{YD}                \boxrowta{8/LimeGreen/,3/white/,1/white/\ast}         \boxrow{6/Goldenrod/,3/white/}            \boxrow{6/LimeGreen/,2/white/,1/white/\ast}           \boxrow{4/Goldenrod/,2/white/}               \boxrow{3/Goldenrod/,2/white/}                \boxrow{1/Aquamarine/,1/white/,1/white/\ast}             \boxrow{1/Aquamarine/,1/white/\ast}
                \boxrow{1/Goldenrod/}
            \end{YD}
        };
        \end{tikzpicture}
        \]
        
        \begin{remark}
            Adding $0$ does not change the value of the parts in $\lambda^{(3)},$ so we do not have to draw it. However, since there is $\ast$ on the part of size $0,$ we draw it to remember to overline $\lambda_1^{(3)}$ (see below).
        \end{remark}
        
        Finally, we extend the color of each row, we overline each part that is immediately below a part ending with $\ast$ , and we remove those $\ast$ :
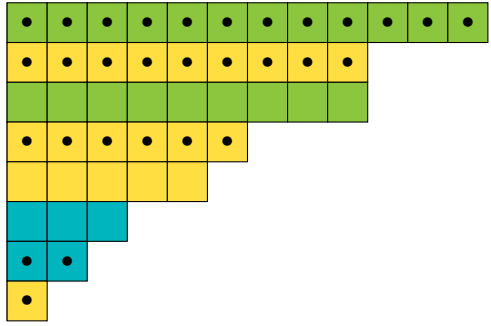
        \[
        \ytableausetup{centertableaux}
        \begin{tikzpicture}[baseline=(current bounding box.center)]
        \node (first) {
            \begin{YD}                \boxrow{12/LimeGreen/\bullet}\boxrow{9/Goldenrod/\bullet}   \boxrow{9/LimeGreen/}     \boxrow{6/Goldenrod/\bullet}   \boxrow{5/Goldenrod/}          \boxrow{3/Aquamarine/}            \boxrow{2/Aquamarine/\bullet}     \boxrow{1/Goldenrod/\bullet}
            \end{YD}
        };
        \end{tikzpicture}
        \]
We get the overpartition $\lambda^{(4)} = (\overline{12}_3,\overline{9}_1,9_3,\overline{6}_1,5_1,3_2,\overline{2}_2,\overline{1}_1).$
\end{example}
    
\end{enumerate}
\begin{proposition}
The overpartition \(\lambda^{(4)}\) belongs to \(\mathfrak{S}(x_1,x_2;m,n)\).
\end{proposition}
\begin{proof}\ \\
\begin{enumerate}
            \item[] \hl{It is an overpartition:}
            \;True, since a partial staircase is added in the last step.
            \item[] \hl{There are no parts $1_3$ or $\overline{1}_{3}$:}
            \;Indeed, $\lambda_L^{(2)}$ cannot be $1_3$ (or $\overline{1}_{3}$) since it is untouched by Step 2 and the color $3$ comes with adding $1$ to the size of the part in Step 1. Moreover, $\lambda_{L-1}^{(2)}$ cannot be $1_3$ (or $\overline{1}_{3}$), since it would mean that either $\lambda_{L-1}^{(1)}$ is already $1_3$ (or $\overline{1}_{3}$) (which is impossible for the same reason as for the smallest part), or that $\lambda_{L-1}^{(1)}=2_{3}$ (or $\overline{2}_3$) and $\lambda_L^{(1)}$ is overlined (which is impossible since an overlined part is the first occurence of its size, so $\lambda_{L-1}$ could not have been of size $1$ before Step 1). The same reasoning can be done for the other parts. Furthermore, the insertion in Step 3 does not change the value or color of the parts in $\lambda^{(2)}.$ Only parts from $\mu^{(2)}$ are added, and those parts have color $2,$ so there is no problem here. Finally, adding the partial staircase in Step 4 only enlarges the parts of $\lambda^{(3)},$ so we are done.
            \item[] \hl{There are $x_1$ parts that have $1$ in their color's binary decomposition and $x_2$ parts that have $2$ in their color's binary decomposition:}
    \;Indeed, color $1$ is the original color of $\lambda,$ thus the parts from $\lambda$ who kept their original color through all the steps are still counted by $x_1$ at the end. Color $2$ comes from putting the remaining parts from $\mu^{(2)}$ into $\lambda^{(2)}$ in Step 3, and the color of those parts remains untouched during the complete procedure, so they are still counted by $x_2.$ Finally, parts in the color $3$ come from adding $2$ to the original color of some parts of $\lambda$ during Step 1. Therefore, they are counted by both $x_1$ and $x_2,$ since they each come from a part of $\mu$ added to the side of $\lambda$ and a part of $\lambda$ whose color is changed.
            \item[] \hl{For all i, $\lambda_i^{(4)}-\lambda_{i+1}^{(4)} \geq \omega(c_i^{(4)}) + \delta(c_i^{(4)},c_{i+1}^{(4)}) - \mathbf{1}_{\lambda_{i+1}^{(4)} \text{ is not overlined}}$:}
   \;At the beginning, $\lambda$ is an overpartition, so it satisfies the following gap conditions:
\[
\lambda_i-\lambda_{i+1} \geq 1 - \mathbf{1}_{\lambda_{i+1} \text{ is not overlined}},
\]
for all $i.$
\noindent After Step 1, $\lambda^{(1)}$ satisfies
    \[
    \lambda_i^{(1)}-\lambda_{i+1}^{(1)} \geq \omega(c_i^{(1)}) - \mathbf{1}_{\lambda_{i+1}^{(1)} \text{ is not overlined}},
    \]
for all $i.$ Indeed, if $c_i^{(1)}=3$ then $\lambda_i^{(1)}-\lambda_{i+1}^{(1)}=\lambda_i-\lambda_{i+1}+1.$
\noindent After Step 2, $\lambda^{(2)}$ satisfies
    \[
    \lambda_i^{(2)}-\lambda_{i+1}^{(2)} \geq \omega(c_i^{(2)}) - 1,
    \]
for all $i,$ since removing the partial staircase makes the gap imposed by the overline disappear.
\noindent After Step 3, $\lambda^{(3)}$ satisfies
    \[
    \lambda_i^{(3)}-\lambda_{i+1}^{(3)} \geq
\begin{cases}
0 &,\text{ if }c_i^{(3)}=c_{i+1}^{(3)}=2, \\ 
1 &,\text{ if }c_i^{(3)}\neq2\text{ and }c_{i+1}^{(3)}=2, \\
\omega(c_i^{(3)})-1 &,\text{ otherwise}.
\end{cases}
    \]
for all $i.$ Note that this is equivalent to the following:
    \[
    \lambda_i^{(3)}-\lambda_{i+1}^{(3)} \geq \omega(c_i^{(3)}) + \delta(c_i^{(3)},c_{i+1}^{(3)}) -1,
    \]
for all $i.$
\noindent After Step 4, $\lambda^{(4)}$ satisfies
\[
\lambda_i^{(4)}-\lambda_{i+1}^{(4)} \geq \omega(c_i^{(4)}) + \delta(c_i^{(4)}, c_{i+1}^{(4)}) - \mathbf{1}_{\lambda_{i+1}^{(4)} \text{ is not overlined}},
\]
for all $i,$ since adding the partial staircase to $\lambda^{(3)}$ imposes a gap between $\lambda_i^{(4)}$ and $\lambda_{i+1}^{(4)}$ when $\lambda_{i+1}^{(4)}$ is overlined. 

\noindent Therefore, condition (iii) is true for $\lambda^{(4)}.$
            \item[] \hl{The $s$ smallest parts are overlined, where $s \coloneqq V^{(4)}_{2}$:}
    \;Indeed, at the end of Step 4, the $M$ smallest parts of $\lambda^{(4)}$ are overlined, since $\nu$ is partly made of a staircase $(\,L+M-1,\,L+M-2,\dots,\,L\,)$ and $\lambda^{(4)}$ has $L+M$ parts.
\end{enumerate}
\end{proof}
Each of the four steps is a one-to-one correspondence\footnote{For details on the reversibility of the four steps, the reader is refered to the general case in Section~\ref{subsec:caseKbigger3}.}, which completes the proof of Theorem~\ref{thm:LV} in the case $k=2.$

\subsection{Case $k\geq3$}\label{subsec:caseKbigger3}

With the base case established, we now prove the induction step. The case $k=2$ lays the groundwork for the bijection in the general case, and therefore the general bijection is also divided into four steps.

\medskip

Suppose the statement true for a given natural number $k-1.$ Let us prove it for $k.$ \\
    The right-hand side of \eqref{eq:thm1} can be seen as the product of two generating functions:
    \[
    \frac{(-y_1 q;q)_\infty \cdots (-y_{k-1} q;q)_\infty}{(y_1 d q;q)_\infty} \text{\quad and\quad} (-y_k q;q)_\infty.
    \]
    By the induction hypothesis, the first one is the generating function for overpartitions in $\mathfrak{S}(x_1,\dots,x_{k-1};m,n).$ The second one is the generating function for partitions into distinct parts in the color $2^{k-1}.$
Let $x_1,\dots,x_k,m,n,\tilde{n} \in \mathbb{N}, \;\tilde{n} \leq n.$ Consider an overpartition $\lambda$ of $\tilde{n}$ in $\mathfrak{S}(x_1,\dots,x_k;m,\tilde{n})$ and a partition $\mu$ of $n-\tilde{n}$ into $x_k$ distinct parts in the color $2^{k-1}.$

\begin{example}
Consider $\lambda$ to be the partition $(\overline{12}_3,\overline{9}_1,9_3,\overline{6}_1,5_1,3_2,\overline{2}_2,\overline{1}_1)$ obtained in the example at the end of Step 4 in the proof of Theorem~\ref{thm:Lovejoy2017} (Section~\ref{sec:proofnewthm}). It is in $\mathfrak{S}(6,4;48).$ Consider a partition $\mu= (17_4,12_4,11_4,9_4,5_4,2_4)$ of $56$ into distinct parts in color $4.$
\end{example}

\begin{enumerate}[label=\textbf{Step \arabic*:}]
\item Let $L$ denote the number of parts in the overpartition $\lambda.$ For each part $p$ of $\mu$ such that $p \leq L,$ add $1$ to the first $p$ parts of $\lambda$ and add $2^{k-1}$ to the color of $\lambda_p.$ This gives a new overpartition $\lambda^{(1)}.$ Keeping only the unused parts of $\mu,$ we also get a new partition $\mu^{(1)}.$
\begin{example}
Consider the partitions from the previous example and apply Step 1 to them:
\begin{align*}
	L = 8, \\
	(\lambda,\mu) &= ((\overline{12}_3,\overline{9}_1,9_3,\overline{6}_1,5_1,3_2,\overline{2}_2,\overline{1}_1),(17_4,12_4,11_4,9_4,5_4,{2_4})) \\
	&\mapsto ((\overline{{13}}_3,\overline{{10}}_{{5}},9_3,\overline{6}_1,5_1,3_2,\overline{2}_2,\overline{1}_1),(17_4,12_4,11_4,9_4,{5_4})) \\
	&\mapsto ((\overline{{14}}_3,\overline{{11}}_5,{10}_3,\overline{{7}}_1,{6_5},3_2,\overline{2}_2,\overline{1}_1),(17_4,12_4,11_4,9_4)) \\
	&= (\lambda^{(1)},\mu^{(1)}).
\end{align*}
\end{example}
\begin{lemma}\label{lemma:diffcondstep1}
The following difference condition is satisfied by $\lambda^{(1)}$:
    \begin{equation}\label{eq:diffcondstep1}
    \lambda_i^{(1)}-\lambda_{i+1}^{(1)} \geq \omega(c_i^{(1)}) + \delta^*(c_i^{(1)},c_{i+1}^{(1)}) - \mathbf{1}_{\lambda_{i+1}^{(1)} \text{ is not overlined}},
    \end{equation}
for $i=1,\dots,L-1,$ where
    \[
    \delta^*(c_i,c_{i+1}) \coloneqq \begin{cases}
    \delta(c_i ,c_{i+1}), & \text{if } c_i  < 2^{k-1}, \\
    \delta(c_i -2^{k-1},c_{i+1}),  & \text{if } c_i  > 2^{k-1},
    \end{cases}
    \]
and undefined for $c_i=2^{k-1},$ as no part in $\lambda$ had color $2^{k-1}$ and the new colors are made by adding $2^{k-1}$ to an already positive color.
\end{lemma}
\begin{proof}
By the induction hypothesis, $\lambda$ satisfies $\lambda_i-\lambda_{i+1} \geq \omega(c_i) + \delta(c_i,c_{i+1}) -\mathbf{1}_{\lambda_{i+1} \text{ is not overlined}},$ for $i=1,\dots,L-1,$ so using this and Table~\ref{table1}, we can easily verify \eqref{eq:diffcondstep1}.
\begin{table}[H]
    \centering
    \begin{tabular}{|c|c|c|}
    \hline
      & \textbf{$c_i^{(1)}=c_i$} & \textbf{$c_i^{(1)}=c_i+2^{k-1}$} \\
    \hline
    $\lambda_{i+1}$ is overlined & 
    \begin{minipage}{5cm}
    \vspace{6pt} 
    \begin{itemize}[left=0pt]
        \item $\lambda_i^{(1)}-\lambda_{i+1}^{(1)}=\lambda_i-\lambda_{i+1}$
        \item $\omega(c_i^{(1)})=\omega(c_i)$
        \item $\delta^*(c_i^{(1)},c_{i+1}^{(1)}) = \delta(c_i,c_{i+1})$
        \item $\mathbf{1}_{\lambda_{i+1}^{(1)} \text{ is not overlined}} = 0$
    \end{itemize}
    \vspace{1pt} 
    \end{minipage}
    & \begin{minipage}{5cm}
    \vspace{6pt} 
    \begin{itemize}[left=0pt]
        \item $\lambda_i^{(1)}-\lambda_{i+1}^{(1)}\geq \lambda_i-\lambda_{i+1}+1$
        \item $\omega(c_i^{(1)})=\omega(c_i)+1$
        \item $\delta^*(c_i^{(1)},c_{i+1}^{(1)}) = \delta(c_i,c_{i+1})$
        \item $\mathbf{1}_{\lambda_{i+1}^{(1)} \text{ is not overlined}} = 0$
    \end{itemize}
    \vspace{1pt} 
    \end{minipage} \\
    \hline
    $\lambda_{i+1}$ is not overlined & \begin{minipage}{5cm}
    \vspace{6pt} 
    \begin{itemize}[left=0pt]
        \item $\lambda_i^{(1)}-\lambda_{i+1}^{(1)}=\lambda_i-\lambda_{i+1}$
        \item $\omega(c_i^{(1)})=\omega(c_i)$
        \item $\delta^*(c_i^{(1)},c_{i+1}^{(1)}) = \delta(c_i,c_{i+1})$
        \item $\mathbf{1}_{\lambda_{i+1}^{(1)} \text{ is not overlined}} = 1$
    \end{itemize}
    \vspace{1pt} 
    \end{minipage} & \begin{minipage}{5cm}
    \vspace{6pt} 
    \begin{itemize}[left=0pt]
        \item $\lambda_i^{(1)}-\lambda_{i+1}^{(1)}\geq \lambda_i-\lambda_{i+1}+1$
        \item $\omega(c_i^{(1)})=\omega(c_i)+1$
        \item $\delta^*(c_i^{(1)},c_{i+1}^{(1)}) = \delta(c_i,c_{i+1})$
        \item $\mathbf{1}_{\lambda_{i+1}^{(1)} \text{ is not overlined}} = 1$
    \end{itemize}
    \vspace{1pt} 
    \end{minipage} \\
    \hline
    \end{tabular}
\caption{Note that the values in this table are the same whether $c_{i+1}^{(1)}=c_{i+1}+2^{k-1} \text{ or } c_{i+1}^{(1)}=c_{i+1}.$}\label{table1}
\end{table}
\end{proof}
\begin{lemma}\label{lemma:smallestpart1}
The smallest part in $\lambda^{(1)}$ satisfies
\begin{equation}\label{eq:smallestpart1}
\lambda_L^{(1)}\geq\omega(c_L^{(1)}).
\end{equation}
\end{lemma}
\begin{proof}
Each part in the partition $\lambda$ gets at most one $2^{k-1}$ added to its color, since $\mu$ is made of distinct parts. Therefore, either $\lambda_L$ is not changed at all by the procedure, thus we conclude by induction, or $c_L^{(1)}=c_L+2^{k-1}$ and $\lambda_L^{(1)} = \lambda_L+1 \geq \omega(c_L)+1=\omega(c_L^{(1)}).$
\end{proof}

\item Let $M$ denote the number of parts in the partition $\mu^{(1)}.$ From $\mu^{(1)},$ get $\mu^{(2)}$ by taking away a (colorless) staircase $\nu'$ of size $(\,L+M-1,\,L+M-2,\dots,\,L\,).$ This is possible, since the smallest part in $\mu^{(1)}$ is bigger than $L$ and all its parts are distinct. From $\lambda^{(1)},$ get a partition $\lambda^{(2)}$ that is not overlined anymore by taking away a (colorless) \emph{generalized staircase}, i.e for each overlined part in $\lambda^{(1)},$ remove the overline and remove $1$ from every part above this one. These $1$'s are moved to the generalized staircase as a new part. Call the generalized staircase $\nu''.$
    
    \begin{remark}
        A generalized staircase can have a part of size $0$ (e.g. if the largest part in $\lambda^{(1)}$ is overlined).
    \end{remark}
    
    Since the largest part of $\nu''$ is at most $L-1,$ combine it with $\nu'.$ This gives a \emph{partial staircase} $\nu.$
\begin{example}
Consider the partitions $\lambda^{(1)}$ and $\mu^{(1)}$ obtained in the previous example:
\[
\lambda^{(1)}=(\overline{14}_3,\overline{11}_5,10_3,\overline{7}_1,6_5,3_2,\overline{2}_2,\overline{1}_1)\quad\text{and}\quad\mu^{(1)}=(17_4,12_4,11_4,9_4).
\]
Since $L=8$ and the number of parts in $\mu^{(1)}$ is $M=4,$ remove a staircase $\nu'=(11,10,9,8)$ from $\mu^{(1)}$:
\begin{align*}
\mu^{(1)}&=(17_4,12_4,11_4,9_4) \\
	&\mapsto ((6_4,2_4,2_4,1_4),(11,10,9,8)) \\
	&= (\mu^{(2)},\nu').
\end{align*}
Remove the generalized staircase $\nu''$ from $\lambda^{(1)}$:
\begin{align*}
\lambda^{(1)}&= ({\overline{\textcolor{black}{14}}}_3,\overline{11}_5,10_3,\overline{7}_1,6_5,3_2,\overline{2}_2,\overline{1}_1) \\
	&\mapsto ((14_3,{\overline{\textcolor{black}{11}}}_5,10_3,\overline{7}_1,6_5,3_2,\overline{2}_2,\overline{1}_1),({0})) \\
	&\mapsto (({13}_3,11_5,10_3,{\overline{\textcolor{black}{7}}}_1,6_5,3_2,\overline{2}_2,\overline{1}_1),({1},0)) \\
	&\mapsto (({12}_3,{10}_5,{9}_3,7_1,6_5,3_2,{\overline{\textcolor{black}{2}}}_2,\overline{1}_1),({3},1,0)) \\
	&\mapsto (({11}_3,{9}_5,{8}_3,{6}_1,{5}_5,{2}_2,2_2,{\overline{\textcolor{black}{1}}}_1),({6},3,1,0)) \\
	&\mapsto (({10}_3,{8}_5,{7}_3,{5}_1,{4}_5,{1}_2,{1}_2,1_1),({7},6,3,1,0)) \\
	&= (\lambda^{(2)},\nu'').
\end{align*}
Finally, combine $\nu'$ and $\nu''$:
\begin{align*}
(\nu',\nu'') &= ((11,10,9,8),(7,6,3,1,0)) \\
&\mapsto (11,10,9,8,7,6,3,1,0) \\
&= \nu.
\end{align*}
\end{example}
\begin{lemma}\label{lemma:diffcondstep2}
The following difference condition is satisfied by $\lambda^{(2)}$:
    \begin{equation}\label{eq:diffcondstep2}
    \lambda_i^{(2)}-\lambda_{i+1}^{(2)} \geq \omega(c_i^{(2)})+ \delta^*(c_i^{(2)},c_{i+1}^{(2)})-1,
    \end{equation}
    for $i=1,\dots,L-1.$
\end{lemma}
\begin{proof}
If $\lambda_{i+1}^{(1)}$ was overlined, then at one point during the procedure we removed 1 from $\lambda_i^{(1)}$ without removing anything from $\lambda_{i+1}^{(1)},$ which implies that
    \begin{align*}
    \lambda_i^{(2)}-\lambda_{i+1}^{(2)}
    &=\lambda_i^{(1)}-1-\lambda_{i+1}^{(1)} \\
    &\stackrel{\eqref{eq:diffcondstep1}}{\geq}\omega(c_i^{(1)})+\delta^*(c_i^{(1)},c_{i+1}^{(1)})-\mathbf{1}_{\lambda_{i+1}^{(1)} \text{ is not overlined}}-1 \\
    &=\omega(c_i^{(1)})+\delta^*(c_i^{(1)},c_{i+1}^{(1)})-1 \\
    &=\omega(c_i^{(2)})+\delta^*(c_i^{(2)},c_{i+1}^{(2)})-1.
    \end{align*}
In the other case, if $\lambda_{i+1}^{(1)}$ was not overlined, then the same amount has been taken away from $\lambda_i^{(1)}$ and $\lambda_{i+1}^{(1)}$, thus we have
    \begin{align*}
    \lambda_i^{(2)}-\lambda_{i+1}^{(2)}
    &=\lambda_i^{(1)}-\lambda_{i+1}^{(1)} \\
    &\stackrel{\eqref{eq:diffcondstep1}}{\geq}\omega(c_i^{(1)})+\delta^*(c_i^{(1)},c_{i+1}^{(1)})-\mathbf{1}_{\lambda_{i+1}^{(1)} \text{ is not overlined}} \\
    &=\omega(c_i^{(1)})+\delta^*(c_i^{(1)},c_{i+1}^{(1)})-1 \\
    &=\omega(c_i^{(2)})+\delta^*(c_i^{(2)},c_{i+1}^{(2)})-1.
    \end{align*}
\end{proof}
\begin{lemma}\label{lemma:smallestpart2}
The smallest part in $\lambda^{(2)}$ satisfies
\begin{equation}\label{eq:smallestpart2}
\lambda_L^{(2)}\geq\omega(c_L^{(2)}).
\end{equation}
\end{lemma}
\begin{proof}
When the generalized staircase is removed from $\lambda^{(1)},$ the only possible change to the smallest part is the removal of its overline; its color and size remain unchanged. Therefore,
\[
\lambda_L^{(2)}=\lambda_L^{(1)}\stackrel{\eqref{eq:smallestpart1}}{\geq}\omega(c_L^{(1)})=\omega(c_L^{(2)}).
\]
\end{proof}

\item In this general setting, Step 3 is more complex than in the case $k=2$ treated in Section~\ref{subsec:caseKis2}. Simply inserting the parts of $\mu^{(2)}$ into $\lambda^{(2)}$ does not satisfy the gap conditions, so we must adjust the colors to ensure they are met. This idea comes from the proof of Theorem~\ref{thm:CorteelLovejoy} by Corteel and Lovejoy \cite{CorteelLovejoy2006}, and in fact Step 3 here coincides with the third step of the bijection in that proof.

In this step, the parts of $\mu^{(2)}$ are inserted in $\lambda^{(2)}$ one by one, from largest to smallest, to get a new partition $\lambda^{(3)}.$ Before inserting anything, we define $\lambda^{(3)} \coloneqq \lambda^{(2)}.$ Then, $\lambda^{(3)}$ will grow as more parts are added to it.
    
Start by inserting $\mu_1^{(2)}$ such that it becomes the first occurence of its value in the partition $\lambda^{(3)}.$ Without loss of generality, suppose $\mu_1^{(2)}$ is at the $i$-th position in the partition, i.e
\[
\lambda^{(3)}=(\lambda_1^{(3)}, \dots, \lambda_{i-1}^{(3)}, \mu_1^{(2)}, \lambda_i^{(3)}, \dots, \lambda_{L}^{(3)}).
\]
Rename the parts in $\lambda^{(3)}$ as
\[
\lambda^{(3)}= (\lambda_1^{(3)}, \dots, \lambda_{i-1}^{(3)}, \lambda_i^{(3)}, \lambda_{i+1}^{(3)}, \dots, \lambda_{L+1}^{(3)}),
\]
so that $\mu_1^{(2)}$ is now named $\lambda_i^{(3)}.$ Since it is the first occurrence of its value, $\lambda_{i-1}^{(3)}-\lambda_i^{(3)}\geq1$ and $\lambda_i^{(3)}-\lambda_{i+1}^{(3)}\geq0.$
Furthermore, since $\omega(c_i^{(3)})=\omega(2^{k-1})=1$ and $\delta(c_i^{(3)},c_{i+1}^{(3)})=\delta(2^{k-1},c_{i+1}^{(3)})=0,$ this last inequality is equivalent to
\[
\lambda_i^{(3)}-\lambda_{i+1}^{(3)}\geq\omega(c_i^{(3)})+\delta(c_i^{(3)},c_{i+1}^{(3)})-1.
\]
    However, it may be that
    \begin{equation}\label{eq:diffcondproblemstep3}
    \lambda_{i-1}^{(3)}-\lambda_i^{(3)}<\omega(c_{i-1}^{(3)})+\delta(c_{i-1}^{(3)},c_i^{(3)})-1,
    \end{equation}
    and in this case, a \emph{redistribution of colors} takes place between $\lambda_{i-1}^{(3)}$ and $\lambda_i^{(3)}.$ Specifically, if $\lambda_{i-1}^{(3)} - \lambda_i^{(3)} = j,$ where $j \geq 1$ and $\omega(c_{i-1}^{(3)}) > 1 + j - \delta(c_{i-1}^{(3)}, c_i^{(3)}),$ then we modify the colors of the two parts as follows. From the color $c_{i-1}^{(3)}$, we take the first $j$ smallest powers of $2$ of its binary decomposition and combine them with the color $2^{k-1}$ to obtain $\tilde{c}_{i-1}^{(3)}$, while the remaining portion of $c_{i-1}^{(3)}$ is assigned to $\tilde{c}_i^{(3)}$. The values of the parts themselves remain unchanged, but we denote them by $\tilde{\lambda}_{i-1}^{(3)}$ and $\tilde{\lambda}_i^{(3)}$ to distinguish the situation after the change of colors.

Continue the insertion until all the parts of color $2^{k-1}$ are inserted in $\lambda^{(3)}.$
\begin{example}
Consider the partitions $\lambda^{(2)}=(7_3,6_5,5_3,3_1,3_5,1_2,1_2,1_1)$ and $\mu^{(2)}=(6_4,2_4,2_4,1_4)$ obtained in the previous example.
Insert each part of $\mu^{(2)}$ inside $\lambda^{(2)},$ starting from the largest one, such that it becomes the first occurrence of its value. Observe that there is a case of redistribution of colors when inserting the first part $6_4$:
    \[
        \lambda^{(3)} = (\underbrace{7_3,6_4},6_5,5_3,3_1,3_5,1_2,1_2,1_1),
    \]
Indeed, there is a problem because $7-6=1<2=\omega(3)+\delta(3,4)-1.$ Let us see the details of the redistribution in this case. The binary decomposition of $3$ is $1+2$ and $j=7-6=1.$ Therefore, the new color for the part $7_3$ is the sum of the $j=1$ smallest powers of two in the binary decomposition of $3$, i.e $1,$ to which is added $4.$ This gives the color $1+4=5.$ The new color for the part $6_4$ is the sum of the remaining powers of two in the binary decomposition, i.e $2.$ Then,
    \[
        \lambda^{(3)} = (\underbrace{7_5,6_2},6_5,5_3,3_1,3_5,1_2,1_2,1_1).
    \]
There is no problem anymore since $7-6=1\geq1=\omega(5)+\delta(5,2)-1,$ and continuing the insertion, no more problems are encountered in this example. The partition obtained after Step 3 is
    \[
    \lambda^{(3)}=(7_5,6_2,6_5,5_3,3_1,3_5,2_4,2_4,1_4,1_2,1_2,1_1).
    \]
\end{example}
\begin{remark}
            When $k=2,$ redistributions of colors never occur. Indeed,
            \[
            \lambda_{i-1}^{(3)}-\lambda_{i}^{(3)}\geq\omega(c_{i-1}^{(3)})+\delta(c_{i-1}^{(3)},2)-1=
            \begin{cases}
                \omega(1)+\delta(1,2)-1=1, & \text{ if }c_{i-1}^{(3)}=1, \\
                \omega(3)+\delta(3,2)-1=1, & \text{ if }c_{i-1}^{(3)}=3,
            \end{cases}
            \]
            which is always true, since the part inserted $\lambda_i^{(3)}$ is the first occurrence of its value.
\end{remark}

\begin{lemma}\label{lemma:redistribcolors}
The following are true:
\begin{enumerate}
\item $\tilde{c}_i^{(3)} \neq 2^{k-1},$
\smallskip
\item $v(\tilde{c}_{i-1}^{(3)}) = v(c_{i-1}^{(3)}),\ z(\tilde{c}_i^{(3)}) = z(c_{i-1}^{(3)}),\ \delta(\tilde{c}_{i-1}^{(3)} - 2^{k-1}, \tilde{c}_i^{(3)}) = 1,$
\smallskip
\item $\tilde{\lambda}_{i-1}^{(3)} - \tilde{\lambda}_i^{(3)} = j = \omega(\tilde{c}_{i-1}^{(3)}) + \delta(\tilde{c}_{i-1}^{(3)}, \tilde{c}_i^{(3)}) - 1,$
\smallskip
\item $\tilde{\lambda}_i^{(3)} - \lambda_{i+1}^{(3)} \geq  \omega(\tilde{c}_i^{(3)}) + \delta^*(\tilde{c}_i^{(3)}, c_{i+1}^{(3)}) - 1,$
\smallskip
\item $\lambda_{i-2}^{(3)} - \tilde{\lambda}_{i-1}^{(3)} \geq \omega(c_{i-2}^{(3)}) + \delta^*(c_{i-2}^{(3)}, \tilde{c}_{i-1}^{(3)}) - 1,$
\smallskip
\item Any additional occurrences of $\mu_1^{(2)}$ in color $2^{k-1}$ may now be inserted without violating the conditions.
\item $\tilde{\lambda}_{i-1}^{(3)} - \tilde{\lambda}_i^{(3)} < \omega(\tilde{c}_{i-1}^{(3)}) + \delta(\tilde{c}_{i-1}^{(3)}-2^{k-1}, \tilde{c}_i^{(3)}) - 1.$
\end{enumerate}
\end{lemma}
\begin{proof}\ \\
\begin{enumerate}
\item We have $\tilde{c}_i^{(3)} \neq 2^{k-1},$ otherwise, we would have $\omega(c_{i-1}^{(3)})=j+1$ and $\delta(c_{i-1}^{(3)},c_i^{(3)})=0,$ implying that $\omega(c_{i-1}^{(3)})=j+1-\delta(c_{i-1}^{(3)},c_i^{(3)}),$ which contradicts our hypothesis.
\item By construction.
\item Indeed, $\tilde{c}_{i-1}^{(3)} > 2^{k-1}$ implies that $\delta(\tilde{c}_{i-1}^{(3)}, \tilde{c}_i^{(3)}) = 0,$ and $\omega(\tilde{c}_{i-1}^{(3)})=j+1$ by construction.
\item Since we are in the case \eqref{eq:diffcondproblemstep3} when adding the part $\mu_1^{(2)}$ as $\lambda_i^{(3)},$ we may observe that $c_{i-1}^{(3)} \neq 2^{k-1},$ or else we would have $\lambda_{i-1}^{(3)}-\lambda_i^{(3)}<1+0-1=0$ which is impossible. Moreover, $c_{i+1}^{(3)} = 2^{k-1}$ would contradict the fact that we start with the largest part $\mu_1^{(2)}.$ So we had $\underbrace{\lambda_{i-1}^{(3)} - \lambda_{i+1}^{(3)} \geq \omega(c_{i-1}^{(3)}) + \delta^*(c_{i-1}^{(3)}, c_{i+1}^{(3)}) - 1}_{(\star)} $ according to \eqref{eq:diffcondstep2}. Hence, we may deduce that
    \begin{align*}
    \tilde{\lambda}_i^{(3)} - \lambda_{i+1}^{(3)}
    &= \tilde{\lambda}_i^{(3)} - \lambda_{i-1}^{(3)} + \lambda_{i-1}^{(3)} - \lambda_{i+1}^{(3)} \, , \\
    &\geq \omega(c_{i-1}^{(3)}) - j + \delta^*(c_{i-1}^{(3)}, c_{i+1}^{(3)}) - 1 \, , & \text{by } \tilde{\lambda}_i^{(3)} - \lambda_{i-1}^{(3)} = -j \text{ and } (\star) \, . \\
    &= \omega(\tilde{c}_i^{(3)}) + \delta^*(\tilde{c}_i^{(3)}, c_{i+1}^{(3)}) - 1 \, , & \text{by } z(\tilde{c}_i^{(3)}) = z(c_{i-1}^{(3)}) \text{ and } \omega(\tilde{c}_i^{(3)})=\omega(c_{i-1}^{(3)})-j \, .
    \end{align*}
\item Indeed, $v(\tilde{c}_{i-1}^{(3)}) = v(c_{i-1}^{(3)})$ implies that $\delta(c_{i-2}^{(3)}, c_{i-1}^{(3)})=\delta(c_{i-2}^{(3)}, \tilde{c}_{i-1}^{(3)}),$ and since $\tilde{\lambda}_{i-1}^{(3)} = \lambda_{i-1}^{(3)},$ the minimal difference between $\lambda_{i-2}^{(3)}$ and $\tilde{\lambda}_{i-1}^{(3)}$ is the same as the one between $\lambda_{i-2}^{(3)}$ and $\lambda_{i-1}^{(3)}.$
\item Indeed, suppose we insert a second part $p$ of color $2^{k-1}$ and of the same size as $\mu_1^{(2)}.$ When inserted, it is positioned just above the previous inserted part. To avoid confusion, we will keep the same notations and we will say that the new partition is now of the form $(\lambda_1^{(3)},\dots,\,\lambda_{i-2}^{(3)},\,\tilde{\lambda}_{i-1}^{(3)},\,p,\,\tilde{\lambda}_i^{(3)},\,\lambda_{i+1}^{(3)},\dots)$\;. Then,
    \begin{align*}
        \tilde{\lambda}_{i-1}^{(3)}-p
        &=\tilde{\lambda}_{i-1}^{(3)}-\tilde{\lambda}_i^{(3)} \\
        &=j \\
        &=\omega(\tilde{c}_{i-1}^{(3)}) + \delta(\tilde{c}_{i-1}^{(3)}, \tilde{c}_i^{(3)}) - 1 \\
        &=\omega(\tilde{c}_{i-1}^{(3)}) - 1 \\
        &= \omega(\tilde{c}_{i-1}^{(3)}) + \delta(\tilde{c}_{i-1}^{(3)}, 2^{k-1}) - 1,
    \end{align*}
    and
    \begin{align*}
        p-\tilde{\lambda}_i^{(3)}
        &=0 \\
        &=\omega(2^{k-1}) + \delta(2^{k-1}, \tilde{c}_i^{(3)}) - 1.
    \end{align*}
\item The redistribution of colors implies that $\omega(\tilde{c}^{(3)}_{i-1})=j+1$ and that $\delta(\tilde{c}^{(3)}_{i-1}-2^{k-1}, \tilde{c}_i^{(3)})=1.$ By the point (c) of this lemma, we have $\tilde{\lambda}_{i-1}^{(3)} - \tilde{\lambda}_i^{(3)} = j < j+1 = \omega(\tilde{c}_{i-1}^{(3)}) + \delta(\tilde{c}_{i-1}^{(3)}-2^{k-1}, \tilde{c}_i^{(3)}) - 1.$
\end{enumerate}
\end{proof}
\begin{lemma}\label{lemma:diffcondstep3}
The following difference condition is satisfied by $\lambda^{(3)}$:
\begin{equation}\label{eq:diffcondstep3}
\lambda_i^{(3)}-\lambda_{i+1}^{(3)}\geq\omega(c_i^{(3)})+\delta(c_i^{(3)},c_{i+1}^{(3)})-1,
\end{equation}
for $i=1,\dots,L+M-1.$
\end{lemma}
\begin{proof}
Let $j\in\{1,\dots,L+M-1\}.$ Cases where a redistribution of colors occured between $\lambda_j^{(3)}$ and $\lambda_{j+1}^{(3)}$ have been treated above in point (c) of Lemma~\ref{lemma:redistribcolors}, and we have
\[
\lambda_j^{(3)}-\lambda_{j+1}^{(3)} \geq \omega(c_j^{(3)})+ \delta(c_j^{(3)},c_{j+1}^{(3)})-1.
\]
Otherwise, by \eqref{eq:diffcondstep2} and by points (d) and (e) of Lemma~\ref{lemma:redistribcolors}, we have
\begin{align*}
\lambda_j^{(3)}-\lambda_{j+1}^{(3)}
	&\geq \omega(c_j^{(3)})+ \delta^*(c_j^{(3)},c_{j+1}^{(3)})-1 \\
	&\geq \omega(c_j^{(3)})+ \delta(c_j^{(3)},c_{j+1}^{(3)})-1,
\end{align*}
since $c_j^{(3)} < 2^{k-1}$ is immediate and $c_j^{(3)} > 2^{k-1}$ implies that $\delta^*(c_j^{(3)},c_{j+1}^{(3)}) = \delta(c_j^{(3)}-2^{k-1},c_{j+1}^{(3)}) \geq \delta(c_j^{(3)},c_{j+1}^{(3)}).$
\end{proof}
\begin{lemma}\label{lemma:smallestpart3}
The smallest part in $\lambda^{(3)}$ satisfies
\begin{equation}\label{eq:smallestpart3}
\lambda_{L+M}^{(3)} \geq \omega(c_{L+M}^{(3)}).
\end{equation}
\end{lemma}
\begin{proof}
First case, $p \geq \lambda_L^{(2)}$ for all parts $p$ in $\mu^{(2)}.$ Hence, at the end of Step 3, $\lambda_{L+M}^{(3)}=\lambda_L^{(2)} \stackrel{\eqref{eq:smallestpart2}}{\geq} \omega(c_L^{(2)})=\omega(c_{L+M}^{(3)}).$

Second case, there exist some parts in $\mu^{(2)}$ that are smaller than $\lambda_L^{(2)}.$ Let $p$ be the largest such part. If the inequality \eqref{eq:diffcondproblemstep3} is false when inserting $p,$ then $p$ keeps its original color $2^{k-1}$ and the same is true for the smaller parts (they will all keep their color $2^{k-1}$). That way, at the end of Step 3, $\lambda_{L+M}^{(3)} \geq 1=\omega(2^{k-1})=\omega(c_{L+M}^{(3)}).$ If, on the contrary, \eqref{eq:diffcondproblemstep3} is true when inserting $p,$ then it means that a redistribution of colors occurs. Call $\tilde{p}$ (resp. $\tilde{\lambda}_L^{(2)}$) and $\tilde{c}_p$ (resp. $\tilde{c}_L^{(2)}$) the part $p$ (resp. $\lambda_L^{(2)}$) and its color after the redistribution. Then,
    \[
    \omega(\tilde{c}_p) =\omega(c_L^{(2)})-(\lambda_L^{(2)}-p) \stackrel{\eqref{eq:smallestpart2}}{\leq} p = \tilde{p}.
    \]
The first equality comes from how the colors are redistributed. The last equality is true since the change in colors does not affect the value of the part.

If $p$ is the only part in $\mu^{(2)}$ smaller than $\lambda_L^{(2)},$ we are done. If there are more, the remaining parts can now be inserted without a redistribution of colors taking place. If they are all the same size as $p,$ then $\tilde{p}$ is still the smallest and we are done. If some are strictly smaller than $p,$ then the smallest part has color $2^{k-1}$ and the inequality $\lambda_{L+M}\geq\omega(2^{k-1})$ is always true.
\end{proof}
    
    \item In this last step, we add the partial staircase $\nu$ back to $\lambda^{(3)}$ as if it were a generalized staircase. To do so, for each part $p$ of $\nu,$ add $1$ to the first $p$ parts of $\lambda^{(3)}$ and overline the ($p+1$)-th part. This is possible since $\lambda^{(3)}$ has $L+M$ parts and the size of the largest part in $\nu$ is $L+M-1.$ The partition obtained is denoted $\lambda^{(4)}.$
\begin{example}
Consider the partition $\lambda^{(3)}=(10_3,8_5,7_5,6_2,5_1,4_5,2_4,2_4,1_4,1_2,1_2,1_1)$ obtained in the previous example and the partial staircase $\nu=(11,10,9,8,7,6,3,1,0)$ obtained in the one before that, then add $\nu$ to $\lambda^{(3)}$:
\begin{align*}
(\lambda^{(3)},\nu) &= ((10_3,8_5,7_5,6_2,5_1,4_5,2_4,2_4,1_4,1_2,1_2,1_1),({11},10,9,8,7,6,3,1,0)) \\
	&\mapsto (({11}_3,{9}_5,{8}_5,{7}_2,{6}_1,{5}_5,{3}_4,{3}_4,{2}_4,{2}_2,{2}_2,{\overline{\textcolor{black}{1}}}_1),({10},9,8,7,6,3,1,0)) \\
	&\mapsto (({12}_3,{10}_5,{9}_5,{8}_2,{7}_1,{6}_5,{4}_4,{4}_4,{3}_4,{3}_2,{\overline{\textcolor{black}{2}}}_2,\overline{1}_1),(9,8,7,6,3,1,0)) \\
	&\dots \\
	&\mapsto (\overline{18}_3,\overline{15}_5,14_5,\overline{12}_2,11_1,10_5,\overline{7}_4,\overline{6}_4,\overline{4}_4,\overline{3}_2,\overline{2}_2,\overline{1}_1) \\
	&= \lambda^{(4)}.
\end{align*}
It can easily be verified that $\lambda^{(4)}$ is in $\mathfrak{S}(6,4,6;3,103).$
\end{example}
\begin{lemma}\label{lemma:diffcondstep4}
The following difference condition is satisfied by $\lambda^{(4)}$:
\begin{equation}\label{eq:diffcondstep4}
\lambda_i^{(4)}-\lambda_{i+1}^{(4)}\geq\omega(c_i^{(4)})+\delta(c_i^{(4)},c_{i+1}^{(4)})-\mathbf{1}_{\lambda_{i+1}^{(4)} \text{ is not overlined}},
\end{equation}
for $i=1,\dots,L+M-1.$
\end{lemma}
\begin{proof}
 If $\lambda_{i+1}^{(4)}$ is overlined, then we have
    \begin{align*}
    \lambda_i^{(4)}-\lambda_{i+1}^{(4)}
    &=\lambda_i^{(3)}+1-\lambda_{i+1}^{(3)} \\
    &\stackrel{\eqref{eq:diffcondstep3}}{\geq}\omega(c_i^{(3)})+\delta(c_i^{(3)},c_{i+1}^{(3)}) \\
    &=\omega(c_i^{(4)})+\delta(c_i^{(4)},c_{i+1}^{(4)})-\mathbf{1}_{\lambda_{i+1}^{(4)} \text{ is not overlined}},
    \end{align*}
and if $\lambda_{i+1}^{(4)}$ is not overlined, then we have
    \begin{align*}
    \lambda_i^{(4)}-\lambda_{i+1}^{(4)}
    &=\lambda_i^{(3)}-\lambda_{i+1}^{(3)} \\
    &\stackrel{\eqref{eq:diffcondstep3}}{\geq}\omega(c_i^{(3)})+\delta(c_i^{(3)},c_{i+1}^{(3)})-1 \\
    &=\omega(c_i^{(4)})+\delta(c_i^{(4)},c_{i+1}^{(4)})-\mathbf{1}_{\lambda_{i+1}^{(4)} \text{ is not overlined}}.
    \end{align*}
\end{proof}
\begin{lemma}\label{lemma:smallestpart4}
The smallest part in $\lambda^{(4)}$ satisfies
\begin{equation}\label{eq:smallestpart4}
\lambda_L^{(4)} \geq \omega(c_L^{(4)}).
\end{equation}
\end{lemma}
\begin{proof}
When the partial staircase is added to $\lambda^{(3)},$ the only possible change to the smallest part is the addition of an overline; its color and size remain unchanged. Therefore,
\[
\lambda_L^{(4)}=\lambda_L^{(3)} \stackrel{\eqref{eq:smallestpart3}}{\geq} \omega(c_L^{(3)})=\omega(c_L^{(4)}).
\]
\end{proof}
\end{enumerate}
This completes the procedure. It remains to verify that the resulting overpartition $\lambda^{(4)}$ satisfies all the necessary conditions to belong to $\mathfrak{S}(x_1,\dots,x_k;m,n).$ For this verification, we first establish the following lemma.
\begin{lemma}\label{lemma:fourthcond}
Let $M$ denote the number of parts in the partition $\mu^{(1)}.$ Recall that $s_{k-1}\coloneqq\sum_{r=1}^{k-2} V_{2^r}$ is the minimal number of smallest parts of $\lambda$ which are overlined (from the fourth condition of Definition~\ref{def:Soverlined}). Define $s\coloneqq\sum_{r=1}^{k-1} V^{(4)}_{2^r}.$ Then, $s=s_{k-1}+M.$
\end{lemma}
\begin{proof}
First, note that $V_{2^r}=V_{2^r}^{(1)}=V_{2^r}^{(2)}$ for $r=1,\dots,k-2$ and that $V_{2^r}^{(3)}=V_{2^r}^{(4)}$ for $r=1,\dots,k-1.$ Second, note that $V_{2^{k-1}}^{(4)}$ is equal to the number of parts from $\mu^{(2)}$ inserted in $\lambda^{(2)}$ during Step 3 that did \underline{not} undergo a redistribution of colors. Each pair of parts $(\lambda^{(3)}_{i-1}, \lambda^{(3)}_i)$ that \underline{did} undergo a redistribution of colors are in one of those two forms:
  \begin{enumerate}
	        \item $v(c_{i-1}^{(3)})=1,$ so $v(\tilde{c}_{i-1}^{(3)})=1$ and $2^{k-1}>v(\tilde{c}_i^{(3)})>1.$ This means that $\lambda_{i-1}^{(3)}$ was not counted by any $V_{2^r}$ and $\tilde{\lambda}_{i-1}^{(3)}$ is still not counted by any $V_{2^r}^{(4)}.$ However, $\tilde{\lambda}_i^{(3)},$ that was not in $\lambda^{(3)}$ before the insertion, is counted by $V^{(4)}_{2^t}$ for some $t$ such that $k-1>t>0.$
	        \item $v(c_{i-1}^{(3)})=2^r$ for some $r\geq1,$ so $v(\tilde{c}_{i-1})=2^r$ and $2^{k-1}>v(\tilde{c}_i^{(3)})>2^r.$ This means that $\lambda_{i-1}^{(3)}$ was counted by $V_{2^r}$ and $\tilde{\lambda}_{i-1}^{(3)}$ is still counted by $V_{2^r}^{(4)}.$ However, $\tilde{\lambda}_i^{(3)},$ that was not in $\lambda^{(3)}$ before the insertion, is counted by $V^{(4)}_{2^t}$ for some $t$ such that $k-1>t>r.$
  \end{enumerate}
Therefore, each such pair adds $1$ to the sum and we have
	    \[
	    s_{k-1} + M = \sum_{r=1}^{k-2} V_{2^r} + M = \sum_{r=1}^{k-2} V_{2^r}^{(4)} + V_{2^{k-1}}^{(4)} = \sum_{r=1}^{k-1} V_{2^r}^{(4)} = s.
	    \]
\end{proof}
\begin{proposition}
The overpartition \(\lambda^{(4)}\) belongs to \(\mathfrak{S}(x_1,\dots,x_k;m,n)\).
\end{proposition}
\begin{proof}\ \\
\begin{enumerate}
    \item[] \hl{It is an overpartition of n:} By construction.
    \item[] \hl{It is in $2^k-1$ colors:}
    \;Remember that $\lambda$ was an overpartition in $2^{k-1}-1$ colors and $\mu$ a partition in the color $2^{k-1}.$ By Step 1, we added the color $2^{k-1}$ to some of the already existing colors of $\lambda,$ creating colors in the range $\llbracket 2^{k-1}+1,2^{k}-1 \rrbracket.$ In Step 3, we inserted the parts of color $2^{k-1}$ from $\mu^{(2)}$ in $\lambda^{(2)}.$
    \item[] \hl{It has $m$ non-overlined parts:}
    \;First, note that $\lambda$ has exactly $m$ non-overlined parts by the induction hypothesis, and it is still the case for $\lambda^{(1)}.$ In Step 2, the generalized staircase formed from $\lambda^{(1)}$ has $L-m$ parts (one for each overlined part in $\lambda^{(1)}$). The staircase from $\mu^{(1)}$ has $M$ parts. After Step 3, $\lambda^{(3)}$ has $L+M$ parts. So, when $\nu$ is added back to $\lambda^{(3)}$ in Step 4, the number of parts left non-overlined is exactly $(L+M)-(L-m)-M=m.$
    \item[] \hl{The smallest part satisfies $\lambda_{L+M}^{(4)} \geq \omega(c_{L+M}^{(4)})$:} 
	\;Follows from Lemma~\ref{lemma:smallestpart4}.
    \item[] \hl{There are $x_i$ parts that have $2^{i-1}$ in their color's binary decomposition, for $i=1,\dots,k$:}
    \;By induction, it was true for $\lambda.$ Along the steps, we added $2^{k-1}$ to the color of exactly $x_k$ parts -- partly during Step 1, when each part of $\mu$ added $2^{k-1}$ to the color of one part of $\lambda,$ and the rest during Step 3 when the parts in $\mu^{(2)}$ of color $2^{k-1}$ were inserted.
    \item[] \hl{For $i=1,\dots,L+M-1$, $\lambda_i^{(4)}-\lambda_{i+1}^{(4)} \geq \omega(c_i^{(4)}) + \delta(c_i^{(4)},c_{i+1}^{(4)}) - \mathbf{1}_{\lambda_{i+1}^{(4)} \text{ is not overlined}}$:}
   \;Follows from Lemma~\ref{lemma:diffcondstep4}.
    \item[] \hl{The $s$ smallest parts are overlined, where $s\coloneqq\sum_{r=1}^{k-1} V^{(4)}_{2^r}$:}
    \;By the induction hypothesis, $\lambda$ had its $s_{k-1}\coloneqq\sum_{r=1}^{k-2} V_{2^r}$ smallest parts overlined. Hence, when creating the generalized staircase in Step 2, the $s_{k-1}$ smallest parts induced $s_{k-1}$ parts at the top of the generalized staircase in the form $(\,L-1,\,L-2,\dots,\,L-s_{k-1}\,).$ \\
    Those parts, combined with the $M$ parts created from $\mu^{(1)},$ form a staircase of $s_{k-1}+M$ parts: $(\,L+M-1,\dots,\,L,\,L-1,\dots,\,L-s_{k-1}\,).$ In Step 4, when put back in $\lambda^{(3)},$ this staircase gives an overline to the $s_{k-1}+M=s\,$ smallest parts of $\lambda^{(4)},$ where this last equality follows from Lemma~\ref{lemma:fourthcond}.
\end{enumerate}
\end{proof}
\medskip
\noindent The four steps define a function that takes an overpartition in $\mathfrak{S}(x_1,\dots,x_{k-1};m,\tilde{n})$ and a partition of $n-\tilde{n}$ into $x_k$ distinct parts in the color $2^{k-1},$ and gives an overpartition in $\mathfrak{S}(x_1,\dots,x_k;m,n).$ In fact, it is a bijection. To verify that, we give the inverse function step by step. Let us consider $\hat{\lambda}^{(4)}$ an overpartition counted by $\overline{S}(x_1,\dots,x_k;m,n)$ and denote the number of its parts by $\hat{L}.$

\begin{enumerate}[label=\textbf{Step $\hat{\arabic*}$:}]
    \setcounter{enumi}{4}
    
    \addtocounter{enumi}{-1}\item The removal of a generalized staircase from an overpartition (explained in Step 2) and the addition of a generalized staircase to a partition (explained in Step 4) are inverse operations. Therefore, to reverse Step 4, we remove from $\hat{\lambda}^{(4)}$ its generalized staircase.  We denote this generalized staircase by $\hat{\nu}$ and the remaining partition by $\hat{\lambda}^{(3)}.$

\begin{example}
Consider the overpartition
\[
\hat{\lambda}^{(4)} \coloneqq \lambda^{(4)} = (\overline{18}_3,\overline{15}_5,14_5,\overline{12}_2,11_1,10_5,\overline{7}_4,\overline{6}_4,\overline{4}_4,\overline{3}_2,\overline{2}_2,\overline{1}_1),
\]
counted by $\overline{S}(6,4,6;3,103),$ from the previous example. We remove its generalized staircase:
\begin{align*}
\hat{\lambda}^{(4)} &= ({\overline{\textcolor{black}{18}}}_3,\overline{15}_5,14_5,\overline{12}_2,11_1,10_5,\overline{7}_4,\overline{6}_4,\overline{4}_4,\overline{3}_2,\overline{2}_2,\overline{1}_1) \\
	&\mapsto ((18_3,{\overline{\textcolor{black}{15}}}_5,14_5,\overline{12}_2,11_1,10_5,\overline{7}_4,\overline{6}_4,\overline{4}_4,\overline{3}_2,\overline{2}_2,\overline{1}_1),({0})) \\
	&\mapsto (({17}_3,15_5,14_5,{\overline{\textcolor{black}{12}}}_2,11_1,10_5,\overline{7}_4,\overline{6}_4,\overline{4}_4,\overline{3}_2,\overline{2}_2,\overline{1}_1),({1},0)) \\
	&\mapsto (({16}_3,{14}_5,{13}_5,12_2,11_1,10_5,{\overline{\textcolor{black}{7}}}_4,\overline{6}_4,\overline{4}_4,\overline{3}_2,\overline{2}_2,\overline{1}_1),({3},1,0)) \\
	&\dots  \\
	&\mapsto ((10_3,8_5,7_5,6_2,5_1,4_5,2_4,2_4,1_4,1_2,1_2,1_1),(11,10,9,8,7,6,3,1,0)) \\
	&= (\hat{\lambda}^{(3)},\hat{\nu}).
\end{align*}
\end{example}

\begin{lemma}\label{lemma:hatS}
Let $\hat{s}^{(4)}\coloneqq\sum_{r=1}^{k-1} \hat{V}^{(4)}_{2^r},$ where $\hat{V}^{(4)}_{2^r}$ is the number of parts $\hat{\lambda}^{(4)}_i$ such that $v(\hat{c}^{(4)}_i)=2^r.$ Then,
\begin{enumerate}
\item[(a)] There are at most $\hat{L}$ parts in $\hat{\nu}.$
\item[(b)] The $\hat{s}^{(4)}$ largest parts of $\hat{\nu}$ form an actual staircase: $(\hat{L}-1,\hat{L}-2,\dots,\hat{L}-\hat{s}^{(4)}).$
\end{enumerate}
\end{lemma}
\begin{proof}
\begin{enumerate}
\item[(a)] The generalized staircase $\hat{\nu}$ can have at most the same number of parts as the overpartition $\hat{\lambda}^{(4)}$ it comes from (including a possible part of size zero). This extremal case happens when all parts of the overpartition are overlined.
\item[(b)] Since $\hat{\lambda}^{(4)}$ is counted by $\overline{S}(x_1,\dots,x_k;m,n),$ its $\hat{s}^{(4)}$ smallest parts are overlined, by condition (iv) of Definition~\ref{def:Soverlined}. Therefore, the parts in $\hat{\nu}$ coming from those overlines form a staircase $(\hat{L}-1,\hat{L}-2,\dots,\hat{L}-\hat{s}^{(4)}).$
\end{enumerate}
\end{proof}
\begin{lemma}\label{lemma:diffcondstep5}
The partition $\hat{\lambda}^{(3)}$ satisfies the following difference condition:
    \begin{equation}\label{eq:diffcondstep5}
        \hat{\lambda}_i^{(3)} - \hat{\lambda}_{i+1}^{(3)} \geq \omega(\hat{c}_i^{(3)})+\delta(\hat{c}_i^{(3)},\hat{c}_{i+1}^{(3)}) - 1,
    \end{equation}
    for $i=1,\dots,\hat{L}-1.$
\end{lemma}
\begin{proof}    
To verify \eqref{eq:diffcondstep5}, two cases need to be considered. First case, if $\hat{\lambda}^{(4)}_{i+1}$ was overlined, then at one point during the procedure we removed 1 from $\hat{\lambda}^{(4)}_i$ without removing anything from $\hat{\lambda}^{(4)}_{i+1},$ which implies that
    \begin{align*}
    \hat{\lambda}_i^{(3)}-\hat{\lambda}_{i+1}^{(3)}
    &=\hat{\lambda}^{(4)}_i-1-\hat{\lambda}^{(4)}_{i+1} \\
    &\geq\omega(\hat{c}^{(4)}_i)+\delta(\hat{c}^{(4)}_i,\hat{c}^{(4)}_{i+1})-\mathbf{1}_{\hat{\lambda}^{(4)}_{i+1} \text{ is not overlined}}-1 \\
    &=\omega(\hat{c}^{(4)}_i)+\delta(\hat{c}^{(4)}_i,\hat{c}^{(4)}_{i+1})-1 \\
    &=\omega(\hat{c}_i^{(3)})+\delta(\hat{c}_i^{(3)},\hat{c}_{i+1}^{(3)})-1,
    \end{align*}
where the inequality follows from condition (iii) of Definition~\ref{def:Soverlined}, since $\hat{\lambda}^{(4)}$ is in $\mathfrak{S}(x_1,\dots,x_k;m,n)$.\\
In the other case, if $\hat{\lambda}^{(4)}_{i+1}$ was not overlined, then the same amount has been taken away from $\hat{\lambda}^{(4)}_i$ and $\hat{\lambda}^{(4)}_{i+1}$, thus we have
    \begin{align*}
    \hat{\lambda}_i^{(3)}-\hat{\lambda}_{i+1}^{(3)}
    &=\hat{\lambda}^{(4)}_i-\hat{\lambda}^{(4)}_{i+1} \\
    &\geq\omega(\hat{c}^{(4)}_i)+\delta(\hat{c}^{(4)}_i,\hat{c}^{(4)}_{i+1})-\mathbf{1}_{\hat{\lambda}^{(4)}_{i+1} \text{ is not overlined}} \\
    &=\omega(\hat{c}^{(4)}_i)+\delta(\hat{c}^{(4)}_i,\hat{c}^{(4)}_{i+1})-1 \\
    &=\omega(\hat{c}_i^{(3)})+\delta(\hat{c}_i^{(3)},\hat{c}_{i+1}^{(3)})-1,
    \end{align*}
where the inequality follows from condition (iii) of Definition~\ref{def:Soverlined}, since $\hat{\lambda}^{(4)}$ is in $\mathfrak{S}(x_1,\dots,x_k;m,n)$.
\end{proof}
\begin{lemma}\label{lemma:smallestpart5}
The smallest part of $\hat{\lambda}^{(3)}$ satisfies
\begin{equation}\label{eq:smallestpart5}
\hat{\lambda}_{\hat{L}}^{(3)}\geq\omega(\hat{c}_{\hat{L}}^{(3)}).
\end{equation}
\end{lemma}
\begin{proof}
When we remove a generalized staircase from an overpartition, no changes in size or color are made to the smallest part, so
\[
\hat{\lambda}_{\hat{L}}^{(3)}=\hat{\lambda}^{(4)}_{\hat{L}} \geq \omega(\hat{c}^{(4)}_{\hat{L}})=\omega(\hat{c}_{\hat{L}}^{(3)}),
\]
where the inequality follows from condition (i) of Definition~\ref{def:Soverlined}, since $\hat{\lambda}^{(4)}$ is in $\mathfrak{S}(x_1,\dots,x_k;m,n)$.
\end{proof}

    \addtocounter{enumi}{-2}\item In Step 3, all parts from $\mu^{(2)}$, starting with the largest, are inserted in $\lambda^{(2)}$ as the first occurence of their value. If the inequality \eqref{eq:diffcondproblemstep3} is true when inserting the part, a redistribution of colors takes place. This change of colors is bijective. Indeed, according to point (g) of Lemma~\ref{lemma:redistribcolors}, after a change of colors between two parts $\lambda^{(3)}_{i-1}$ and $\lambda^{(3)}_i$, we have
\[
\tilde{\lambda}^{(3)}_{i-1}-\tilde{\lambda}^{(3)}_i < \omega(\tilde{c}^{(3)}_{i-1}) + \delta^*(\tilde{c}^{(3)}_{i-1},\tilde{c}^{(3)}_{i}) -1,
\]
but if the parts did not undergo a change of colors, we have that the gap condition is greater or equal to the right hand side of this inequality. Therefore, to reverse the process of Step 3, we analyze each part $\hat{\lambda}_i^{(3)}$, starting with the smallest. If the part is in the color $2^{k-1}$, we extract it. If not, we verify
\[
\hat{\lambda}^{(3)}_{i-1} -\hat{\lambda}^{(3)}_i  \geq \omega(\hat{c}_{i-1}^{(3)} )+\delta^*(\hat{c}_{i-1}^{(3)},\hat{c}_{i}^{(3)} )-1.
\]
If this inequality is verified, we proceed to the next part. Note that it is always verified when $\hat{c}_{i-1}^{(3)}<2^{k-1}$, by Lemma~\ref{lemma:diffcondstep5}. If the inequality is not verified, we change the colors as such: name $\tilde{\tilde{\lambda}}_{i-1}^{(3)} ,\tilde{\tilde{\lambda}}_{i}^{(3)} $ the new parts (although their values are still the same) and $\tilde{\tilde{c}}_{i-1}^{(3)} ,\tilde{\tilde{c}}_{i}^{(3)} $ their respective colors, with $\tilde{\tilde{c}}_{i-1}^{(3)}  \coloneqq \hat{c}_{i-1}^{(3)} +\hat{c}_i^{(3)} -2^{k-1}$ and $\tilde{\tilde{c}}_{i}^{(3)}  \coloneqq 2^{k-1}.$

Note that the powers of $2$ in the binary decomposition of $\hat{c}_i^{(3)}$ and $\hat{c}_{i-1}^{(3)}-2^{k-1}$ are all different (see Lemma~\ref{lemma:distinctpowersof2} below) and are only moved from one color to the other during the redistribution, thus the total number of appearance of each $2^r,\,r=1,\dots,k-1,$ is still the same.
\begin{lemma}\label{lemma:distinctpowersof2}
Let $i\in\{2,\dots,L\}.$ If $\hat{c}_{i-1}^{(3)} >2^{k-1},\ \hat{c}_i^{(3)} \neq 2^{k-1}$ and
\[
\begin{cases}
&\hat{\lambda}_{i-1}^{(3)} -\hat{\lambda}_i^{(3)}  \geq \omega(\hat{c}_{i-1}^{(3)} )+\delta(\hat{c}_{i-1}^{(3)},\hat{c}_i^{(3)} )-1,\\
&\hat{\lambda}_{i-1}^{(3)} -\hat{\lambda}_i^{(3)}  < \omega(\hat{c}_{i-1}^{(3)} )+\delta(\hat{c}_{i-1}^{(3)} -2^{k-1},\hat{c}_i^{(3)} )-1,
\end{cases}
\]
then,
\begin{enumerate}
\item[1.] $z(\hat{c}_{i-1}^{(3)} -2^{k-1})<v(\hat{c}_i^{(3)}),$
\item[2.] $\hat{\lambda}_{i-1}^{(3)} -\hat{\lambda}_i^{(3)} = \omega(\hat{c}_{i-1}^{(3)} )-1,$
\item[3.] $\tilde{\tilde{\lambda}}_{i-1}^{(3)} -\tilde{\tilde{\lambda}}_i^{(3)}  < \omega(\tilde{\tilde{c}}_{i-1}^{(3)} )+\delta(\tilde{\tilde{c}}_{i-1}^{(3)},\tilde{\tilde{c}}_i^{(3)} )-1.$
\end{enumerate}
\end{lemma}
\begin{proof}
\begin{enumerate}
\item[1.] Since $\hat{c}_{i-1}^{(3)} >2^{k-1},$ we have $\delta(\hat{c}_{i-1}^{(3)},\hat{c}_i^{(3)})=0.$ Hence, $\delta(\hat{c}_{i-1}^{(3)} -2^{k-1},\hat{c}_i^{(3)})$ must be equal to $1$, i.e $z(\hat{c}_{i-1}^{(3)} -2^{k-1})<v(\hat{c}_i^{(3)}).$
\item[2.] Because of the first point, $\omega(\hat{c}_{i-1}^{(3)} )-1 \leq \hat{\lambda}_{i-1}^{(3)} -\hat{\lambda}_i^{(3)} < \omega(\hat{c}_{i-1}^{(3)} ).$
\item[3.] First case, $2^{k-1}$ is in the binary decomposition of $\hat{c}_i^{(3)}$. Then, $\delta(\tilde{\tilde{c}}_{i-1}^{(3)},\tilde{\tilde{c}}_{i}^{(3)})=0$ and $\omega(\hat{c}_i^{(3)})\geq2,$ as $\hat{c}_i^{(3)}\neq2^{k-1}$. Hence, $\omega(\tilde{\tilde{c}}_{i-1}^{(3)})>\omega(\hat{c}_{i-1}^{(3)}).$ Because of the second point of this lemma, we have the strict inequality. Second case, $2^{k-1}$ is not in the binary decomposition of $\hat{c}_i^{(3)}$. Then, $\delta(\tilde{\tilde{c}}_{i-1}^{(3)},\tilde{\tilde{c}}_{i}^{(3)})=1$ and $\omega(\hat{c}_i^{(3)})\geq1.$ Hence, $\omega(\tilde{\tilde{c}}_{i-1}^{(3)})\geq\omega(\hat{c}_{i-1}^{(3)}).$ Because of the second point of this lemma, we have the strict inequality. 
\end{enumerate}
\end{proof}
By the third point of Lemma~\ref{lemma:distinctpowersof2} and by Lemma~\ref{lemma:diffcondstep6} below, we have that the change of colors is indeed bijective.

All the extracted parts are used to create a partition $\hat{\mu}^{(2)}$ which is entirely colored with $2^{k-1}.$ The remaining parts in $\hat{\lambda}^{(3)}$ also form a partition, denoted $\hat{\lambda}^{(2)}.$

\begin{example}
Consider the partition $\hat{\lambda}^{(3)}=(10_3,8_5,7_5,6_2,5_1,4_5,2_4,2_4,1_4,1_2,1_2,1_1)$ obtained in the previous example. The smallest parts $(1_2,1_2,1_1)$ do not change, and the parts $(2_4,2_4,1_4)$ are all extracted and moved to $\hat{\mu}^{(2)}$:
\[
\hat{\lambda}^{(2)}=(10_3,8_5,7_5,6_2,5_1,4_5,1_2,1_2,1_1)\quad\text{and}\quad\hat{\mu}^{(2)}=(2_4,2_4,1_4).
\]
The pair $(4_5,1_2)$ is such that the color of the first part is larger than $2^{k-1}=4$, but it does not require a redistribution of colors, since $4-1=3\geq2=\omega(5)+\delta(5-4,2)-1.$ The part $5_1$ remains the same.

The interesting pair to look at is the following: $({7_5},{6_2}).$ The color of the first part is larger than $4$, and $7-6=1<2=\omega(5)+\delta(5-4,2)-1.$ Therefore, a redistribution of colors occurs. The new color for the part $7_5$ is $5+2-4=3.$ For the part $6_2,$ the new color is $4$:
\[
\hat{\lambda}^{(2)}=(10_3,8_5,\underbrace{7_3,6_4},5_1,4_5,1_2,1_2,1_1)\quad\text{and}\quad\hat{\mu}^{(2)}=(2_4,2_4,1_4).
\]
Finally, the part $6_4$ is extracted and moved to $\hat{\mu}^{(2)}$ since it is in the color $4.$ The pair $(8_5,7_3)$ is again a case where the color of the first part is larger than $4$, but there is no change of colors, since $8-7=1=\omega(5)+\delta(5,3)-1.$ The largest parts, $(10_3,8_5),$ remain also unchanged. At the end of this step, we have the partitions:
\[
\hat{\lambda}^{(2)}=(10_3,8_5,7_3,5_1,4_5,1_2,1_2,1_1),\ \hat{\mu}^{(2)}=(6_4,2_4,2_4,1_4)\ \text{and}\ \hat{\nu}=(11,10,9,8,7,6,3,1,0).
\]
\end{example}

Let $\hat{M}$ denote the number of parts in $\hat{\mu}^{(2)}.$
\begin{lemma}\label{lemma:diffcondstep6}
The partition $\hat{\lambda}^{(2)}$ satisfies the following difference condition:
    \begin{equation}\label{eq:diffcondstep6}
    \hat{\lambda}_i^{(2)}-\hat{\lambda}_{i+1}^{(2)} \geq \omega(\hat{c}_i^{(2)})+\delta^*(\hat{c}_i^{(2)},\hat{c}_{i+1}^{(2)})-1,
    \end{equation}
    for $i=1,\dots,\hat{L}-\hat{M}-1.$
\end{lemma}
\begin{proof}
    We start with the smallest part $\hat{\lambda}_L^{(3)}.$
    \begin{enumerate}
        \item If $\hat{c}_{\hat{L}}^{(3)} =2^{k-1},$ we extract the part $\hat{\lambda}_{\hat{L}}^{(3)}$ and put it in $\hat{\mu}^{(2)}.$ We then proceed to the next part, $\hat{\lambda}_{\hat{L}-1}^{(3)},$ treating it as the new smallest part.
        \item If $\hat{c}_{\hat{L}-1}^{(3)} >2^{k-1}$ and $\hat{c}_{\hat{L}}^{(3)} \neq2^{k-1},$ we check whether
\begin{equation}\label{eq:diffcondproblemstep6}
\hat{\lambda}^{(3)}_{\hat{L}-1} -\hat{\lambda}^{(3)}_{\hat{L}}  \geq \omega(\hat{c}_{\hat{L}-1}^{(3)} )+\delta(\hat{c}_{\hat{L}-1}^{(3)} -2^{k-1},\hat{c}_{\hat{L}}^{(3)} )-1.
\end{equation}
\begin{enumerate}
            \item[(b.i)] If $\eqref{eq:diffcondproblemstep6}$ is false, we redistribute the colors between the two parts, following the method stated above. Since $\tilde{\tilde{c}}_{\hat{L}}^{(3)} =2^{k-1},$ we extract $\tilde{\tilde{\lambda}}_{\hat{L}}^{(3)} $ and put it in $\hat{\mu}^{(2)}.$ The condition in $\eqref{eq:diffcondstep5}$ is still true for $i=\hat{L}-2$:
            \begin{align*}
            \hat{\lambda}_{\hat{L}-2}^{(3)} -\tilde{\tilde{\lambda}}_{\hat{L}-1}^{(3)} 
            &= \hat{\lambda}_{\hat{L}-2}^{(3)} -\hat{\lambda}_{\hat{L}-1}^{(3)}  \\
            &\stackrel{\eqref{eq:diffcondstep5}}{\geq} \omega(\hat{c}_{\hat{L}-2}^{(3)} ) + \delta(\hat{c}_{\hat{L}-2}^{(3)} ,\hat{c}_{\hat{L}-1}^{(3)} ) - 1 \\
            &\geq \omega(\hat{c}_{\hat{L}-2}^{(3)} ) + \delta(\hat{c}_{\hat{L}-2}^{(3)} ,\tilde{\tilde{c}}_{\hat{L}-1}^{(3)} ) - 1,
            \end{align*}
            since $v(\tilde{\tilde{c}}_{\hat{L}-1}^{(3)} ) \leq v(\hat{c}_{\hat{L}-1}^{(3)} ).$
            We proceed to the next part $\tilde{\tilde{\lambda}}_{\hat{L}-1},$ treating it as the new smallest part.
            
            \item[(b.ii)] If $\eqref{eq:diffcondproblemstep6}$ is true, then so is $\eqref{eq:diffcondstep6}$ for $i=\hat{L}-1.$ We proceed to the next part $\hat{\lambda}_{\hat{L}-1}^{(3)} $ and analyze it the same way. However, if a redistribution of colors is needed between $\hat{\lambda}_{\hat{L}-2}^{(3)} $ and $\hat{\lambda}_{\hat{L}-1}^{(3)} ,$ i.e $\hat{c}_{\hat{L}-2}^{(3)}>2^{k-1}$ and $\hat{\lambda}_{\hat{L}-2}^{(3)}-\hat{\lambda}_{\hat{L}-1}^{(3)}<\omega(\hat{c}_{\hat{L}-2}^{(3)})+\delta(\hat{c}_{\hat{L}-2}^{(3)}-2^{k-1},\hat{c}_{\hat{L}-1}^{(3)})-1,$ it is additionally needed to verify that $\eqref{eq:diffcondstep6}$ works for $\tilde{\tilde{\lambda}}_{\hat{L}-2}^{(3)} $ and $\hat{\lambda}_{\hat{L}}^{(3)} ,$ since $\tilde{\tilde{\lambda}}_{\hat{L}-1}^{(3)}$ will be extracted:
            \begin{align*}
                \tilde{\tilde{\lambda}}_{\hat{L}-2}^{(3)} -\hat{\lambda}_{\hat{L}}^{(3)} 
                &=\hat{\lambda}_{\hat{L}-2}^{(3)} -\hat{\lambda}_{\hat{L}}^{(3)} \\
                &=(\hat{\lambda}_{\hat{L}-2}^{(3)} -\hat{\lambda}_{\hat{L}-1}^{(3)} )+(\hat{\lambda}_{\hat{L}-1}^{(3)} -\hat{\lambda}_{\hat{L}}^{(3)} ) \\
                &\geq (\omega(\hat{c}_{\hat{L}-2}^{(3)} ) + \delta(\hat{c}_{\hat{L}-2}^{(3)} ,\hat{c}_{\hat{L}-1}^{(3)} ) - 1) + (\omega(\hat{c}_{\hat{L}-1}^{(3)} ) + \delta(\hat{c}_{\hat{L}-1}^{(3)} -2^{k-1},\hat{c}_{\hat{L}}^{(3)} ) - 1) \\
                &= \omega(\hat{c}_{\hat{L}-2}^{(3)} )+\omega(\hat{c}_{\hat{L}-1}^{(3)} )-1 + \delta(\hat{c}_{\hat{L}-1}^{(3)} -2^{k-1},\hat{c}_{\hat{L}}^{(3)} ) - 1 \\
                &\geq \omega(\tilde{\tilde{c}}_{\hat{L}-2}^{(3)} )+\delta(\tilde{\tilde{c}}_{\hat{L}-2}^{(3)} -2^{k-1},\hat{c}_{\hat{L}}^{(3)} )-1,
            \end{align*}
            where the third line follows from $\eqref{eq:diffcondstep5}$ applied to $\hat{\lambda}_{\hat{L}-2}^{(3)} -\hat{\lambda}_{\hat{L}-1}^{(3)} $ and $\eqref{eq:diffcondstep6}$ applied to $\hat{\lambda}_{\hat{L}-1}^{(3)} -\hat{\lambda}_{\hat{L}}^{(3)} .$ The fourth line follows from $\hat{c}_{\hat{L}-2}^{(3)} >2^{k-1}.$ Finally, the last line follows from $\omega(\hat{c}_{\hat{L}-2}^{(3)} )+\omega(\hat{c}_{\hat{L}-1}^{(3)} )-1 = \omega(\tilde{\tilde{c}}_{\hat{L}-2}^{(3)} )$ and $z(\hat{c}_{\hat{L}-1}^{(3)} -2^{k-1})\leq z(\tilde{\tilde{c}}_{\hat{L}-2}^{(3)} -2^{k-1}).$
        \end{enumerate}
        
        \item If $\hat{c}_{\hat{L}-1}^{(3)} <2^{k-1}$ and $\hat{c}_{\hat{L}}^{(3)} \neq2^{k-1},$ then $\eqref{eq:diffcondstep6}$ is true for $i=\hat{L}-1,$ by $\eqref{eq:diffcondstep5}.$ We proceed to the next part $\hat{\lambda}_{\hat{L}-1}^{(3)} .$ If a redistribution of colors is needed between $\hat{\lambda}_{\hat{L}-2}^{(3)} $ and $\hat{\lambda}_{\hat{L}-1}^{(3)} ,$ it is additionally needed to verify that $\eqref{eq:diffcondstep6}$ works for $\tilde{\tilde{\lambda}}_{\hat{L}-2}^{(3)} $ and $\hat{\lambda}_{\hat{L}}^{(3)} .$ The argument is the same as above in the case (b.ii).
        
        \item If $\hat{c}_{\hat{L}-1}^{(3)} =2^{k-1}$ and $\hat{c}_{\hat{L}}^{(3)} \neq2^{k-1},$ then $\hat{\lambda}_{\hat{L}-1}^{(3)} $ is removed and put in $\hat{\mu}^{(2)}.$ Suppose $\hat{c}_{\hat{L}-2}^{(3)}  \neq 2^{k-1}$ (otherwise, we iterate the argument for the smallest part above $\hat{\lambda}_{\hat{L}}^{(3)} $ that is not in the color $2^{k-1}$). Then $\eqref{eq:diffcondstep5}$ still stands for $\hat{\lambda}_{\hat{L}-2}^{(3)} $ and $\hat{\lambda}_{\hat{L}}^{(3)} $:
        \begin{align*}
                \hat{\lambda}_{\hat{L}-2}^{(3)} -\hat{\lambda}_{\hat{L}}^{(3)} 
                &= (\hat{\lambda}_{\hat{L}-2}^{(3)} -\hat{\lambda}_{\hat{L}-1}^{(3)} )+(\hat{\lambda}_{\hat{L}-1}^{(3)} -\hat{\lambda}_{\hat{L}}^{(3)}) \\
                &\stackrel{\eqref{eq:diffcondstep5}}{\geq} (\omega(\hat{c}_{\hat{L}-2}^{(3)} ) + \delta(\hat{c}_{\hat{L}-2}^{(3)} ,2^{k-1}) - 1) + (\omega(2^{k-1}) + \delta(2^{k-1},\hat{c}_{\hat{L}}^{(3)} ) - 1) \\
                &= \omega(\hat{c}_{\hat{L}-2}^{(3)} )+\delta(\hat{c}_{\hat{L}-2}^{(3)} ,2^{k-1}) - 1 \\
                &\geq \omega(\hat{c}_{\hat{L}-2}^{(3)} )+\delta(\hat{c}_{\hat{L}-2}^{(3)} ,\hat{c}_{\hat{L}}^{(3)} ) - 1.
            \end{align*}
            We analyze all cases once again, this time with $\hat{\lambda}_{\hat{L}-2}^{(3)}$ instead of $\hat{\lambda}_{\hat{L}-1}^{(3)}.$
    \end{enumerate}
    We continue this procedure step by step, each time analyzing all possible cases, from the smallest part to the largest part of $\hat{\lambda}^{(3)}.$ At the end of Step $\hat{3}$, the parts satisfy the required gap condition. Indeed, if $\hat{c}_i^{(2)}<2^{k-1}$ then the pair $(\hat{\lambda}_i^{(2)},\hat{\lambda}_{i+1}^{(2)})$ satisfies \eqref{eq:diffcondstep5} which is equal to \eqref{eq:diffcondstep6} in this case, and if $\hat{c}_i^{(2)}>2^{k-1}$ then the pair $(\hat{\lambda}_i^{(2)},\hat{\lambda}_{i+1}^{(2)})$ necessarily satisfies $\hat{\lambda}^{(2)}_{i} -\hat{\lambda}^{(2)}_{i+1}  \geq \omega(\hat{c}_{i}^{(2)} )+\delta(\hat{c}_{i}^{(2)} -2^{k-1},\hat{c}_{i+1}^{(2)} )-1$ because of the redistribution of colors.
\end{proof}
\begin{lemma}\label{lemma:smallestpart6}
The smallest part of $\hat{\lambda}^{(2)}$ satisfies
\begin{equation}\label{eq:smallestpart6}
\hat{\lambda}^{(2)}_{\hat{L}-\hat{M}} \geq \omega(\hat{c}^{(2)}_{\hat{L}-\hat{M}}).
\end{equation}
\end{lemma}
\begin{proof}
Various cases have to be considered.
    \begin{enumerate}
        \item[--] First case, $\hat{\lambda}_{\hat{L}}^{(3)}$ is not extracted nor modified by a redistribution of colors. Then $\hat{\lambda}_{\hat{L}-\hat{M}}^{(2)}=\hat{\lambda}_{\hat{L}}^{(3)}\overset{\eqref{eq:smallestpart5}}{\geq}\omega(\hat{c}_{\hat{L}}^{(3)})=\omega(\hat{c}_{\hat{L}-\hat{M}}^{(2)}).$
        \item[--] Second case, $\hat{\lambda}_{\hat{L}}^{(3)}$ undergoes a redistribution of colors. This could happen between $\hat{\lambda}_i^{(3)}$ and $\hat{\lambda}_{\hat{L}}^{(3)}$ for any $i\leq \hat{L}-1,$ if the parts in between have been extracted. However, \eqref{eq:diffcondstep5} is still true for these two parts, as seen in the proof of Lemma~\ref{lemma:diffcondstep6}. So, without loss of generality, suppose it happens between $\hat{\lambda}_{\hat{L}-1}^{(3)}$ and $\hat{\lambda}_{\hat{L}}^{(3)}.$ Then $\tilde{\tilde{c}}_{\hat{L}-1}^{(3)}=\hat{c}_{\hat{L}-1}^{(3)}+\hat{c}_{\hat{L}}^{(3)}-2^{k-1}$ and $\tilde{\tilde{c}}_{\hat{L}}^{(3)}=2^{k-1}.$ This means that $\tilde{\tilde{\lambda}}_{\hat{L}}^{(3)}$ is extracted. The new smallest part is $\tilde{\tilde{\lambda}}_{\hat{L}-1}^{(3)}$ and it satisfies
        \begin{align*}
            \tilde{\tilde{\lambda}}_{\hat{L}-1}^{(3)}=\hat{\lambda}_{\hat{L}-1}^{(3)}
            &\geq \hat{\lambda}_{\hat{L}}^{(3)}+\omega(\hat{c}_{\hat{L}-1}^{(3)})+\delta(\hat{c}_{\hat{L}-1}^{(3)},\hat{c}_{\hat{L}}^{(3)})-1
            &\eqref{eq:diffcondstep5}, \\
            &\geq \omega(\hat{c}_{\hat{L}}^{(3)})+\omega\hat{c}_{\hat{L}-1}^{(3)})+\delta(\hat{c}_{\hat{L}-1}^{(3)},\hat{c}_{\hat{L}}^{(3)})-1
            &(\hat{\lambda}_{\hat{L}}^{(3)}\geq\omega(\hat{c}_{\hat{L}}^{(3)})), \\
            &\geq \omega(\hat{c}_{\hat{L}}^{(3)})+\omega(\hat{c}_{\hat{L}-1}^{(3)})-1
            &(\delta(\cdot,\cdot)\geq0), \\
            &= \omega(\tilde{\tilde{c}}_{\hat{L}-1}^{(3)})
            &(\tilde{\tilde{c}}_{\hat{L}-1}^{(3)}=\hat{c}_{\hat{L}-1}^{(3)}+\hat{c}_{\hat{L}}^{(3)}-2^{k-1}).
        \end{align*}
        In conclusion, $\hat{\lambda}_{\hat{L}-\hat{M}}^{(2)}=\tilde{\tilde{\lambda}}_{\hat{L}-1}^{(3)}\geq\omega(\tilde{\tilde{c}}_{\hat{L}-1}^{(3)})=\omega(\hat{c}_{\hat{L}-\hat{M}}^{(2)}).$
        \item[--] Third case, $\hat{c}_i^{(3)}=\dots=\hat{c}_{\hat{L}}^{(3)}=2^{k-1}$ for some $i\leq \hat{L}.$ Then all associated parts are going to be extracted and the new smallest part will be $\hat{\lambda}_{i-1}^{(3)},$ thus it has to verify $\hat{\lambda}_{i-1}^{(3)}\geq\omega(\hat{c}_{i-1}^{(3)}).$ Suppose by contradiction that $\underbrace{\hat{\lambda}_{i-1}^{(3)}<\omega(\hat{c}_{i-1}^{(3)})}_{(\star\star)}$. Then,
        \[
        \hat{\lambda}_i^{(3)} \stackrel{\eqref{eq:diffcondstep5}}{\leq} \hat{\lambda}_{i-1}^{(3)}-\omega(\hat{c}_{i-1}^{(3)})-\delta(\hat{c}_{i-1}^{(3)},\hat{c}_i^{(3)})+1 \stackrel{(\star\star)}{<} -\delta(\hat{c}_{i-1}^{(3)},\hat{c}_i^{(3)})+1 \leq 1.\ \lightning
        \]
        Therefore, $\hat{\lambda}_{i-1}^{(3)}\geq\omega(\hat{c}_{i-1}^{(3)}).$ In conclusion, $\hat{\lambda}_i,\dots,\hat{\lambda}_{\hat{L}}$ can be extracted and $\hat{\lambda}_{\hat{L}-\hat{M}}^{(2)}=\hat{\lambda}_{i-1}^{(3)}\geq\omega(\hat{c}_{i-1}^{(3)})=\omega(\hat{c}_{\hat{L}-\hat{M}}^{(2)}).$
        \begin{remark}
            If $\hat{\lambda}_{i-1}^{(3)}$ also undergoes a redistribution of colors, repeat the argument of the second case before concluding.
        \end{remark}
    \end{enumerate}
\end{proof}
\begin{lemma}\label{lemma:MsmallerS}
The following inequality is true:
\[
\hat{M}\leq \hat{s}^{(4)},
\]
where \(\hat{s}^{(4)}\) is defined as the sum $\sum_{r=1}^{k-1} \hat{V}^{(4)}_{2^r}$, where \(\hat{V}^{(4)}_{2^r}\) denotes the number of parts \(\hat{\lambda}_i^{(4)}\) such that \(v(\hat{c}^{(4)}_i)=2^r\).
\end{lemma}
\begin{proof}
First, note that $\hat{V}_{2^r}^{(3)}=\hat{V}_{2^r}^{(4)}$ for $r=1,\dots,k-1.$ Second, note that $\hat{M}$ is equal to the sum of $\hat{V}^{(4)}_{2^{k-1}}$ and the number of parts in $\hat{\mu}^{(2)}$ coming from the redistributions of colors. This last number is smaller or equal to $\sum_{r=1}^{k-2} \hat{V}^{(4)}_{2^r}.$ Indeed, each pair $(\hat{\lambda}^{(3)}_{i-1},\hat{\lambda}^{(3)}_i)$ undergoing a redistribution of colors can only be in one of these two forms (see Lemma~\ref{lemma:distinctpowersof2}):
\begin{enumerate}
\item[--] $v(\hat{c}^{(3)}_{i-1})=1$ and $v(\hat{c}^{(3)}_i)=2^r$ with $r>0$. In this case, the part $\hat{\lambda}^{(3)}_i$ is counted once by $\hat{V}_{2^r}^{(3)}.$ After the redistribution, $v(\tilde{c}^{(3)}_{i-1})=1$ and the second part is extracted, hence $\tilde{\lambda}^{(3)}_{i-1}$ will not be involved in another redistribution.
\item[--] $v(\hat{c}^{(3)}_{i-1})=2^t$ and $v(\hat{c}^{(3)}_i)=2^r$ with $r>t>0$. In this case, both parts are counted exactly once, by $\hat{V}_{2^t}^{(3)}$ and $\hat{V}_{2^r}^{(3)}$ respectively. After the redistribution, $v(\tilde{c}^{(3)}_{i-1})=2^t$ and the second part is extracted. The part $\tilde{\lambda}^{(3)}_{i-1}$ could be involved in a second redistribution of colors. In that case, the other part in this new pair would be such that $v(\hat{c}^{(3)}_{i-2})=2^q$ with $r\neq q$, so it would be counted by $\hat{V}_{2^q}^{(3)}$ in the sum. 
\end{enumerate}
Therefore, the number of such pairs is smaller or equal to $\sum_{r=1}^{k-2} \hat{V}^{(4)}_{2^r}.$
\end{proof}
\begin{remark}\label{rem:condition4}
Note that the proof of Lemma~\ref{lemma:MsmallerS} implies that $\sum_{r=1}^{k-2} \hat{V}^{(2)}_{2^r} = \hat{s}^{(4)} - \hat{M}.$
\end{remark}  

\addtocounter{enumi}{-2}\item As said in Step $\hat{4}$, the removal of a generalized staircase from an overpartition and the addition of a generalized staircase to a partition are inverse operations. Similarly, the removal of a staircase from a partition into distinct parts and the addition of a staircase to a partition are inverse operations. Therefore, to reverse Step 2, we add $\hat{\nu}$ partially to $\hat{\lambda}^{(2)}$ and partially to $\hat{\mu}^{(2)}.$

To do so, we take the $\hat{M}$ largest parts in the generalized staircase $\hat{\nu}.$ Since $\hat{M}\leq \hat{s}^{(4)}$ by Lemma~\ref{lemma:MsmallerS}, these parts form a staircase: $(\hat{L}-1,\hat{L}-2,\dots,\hat{L}-\hat{M})$ (see Lemma~\ref{lemma:hatS}). We add those parts to the partition $\hat{\mu}^{(2)}$ which gives $\hat{\mu}^{(1)}\coloneqq(\hat{L}-1+\hat{\mu}_1^{(2)},\,\dots,\hat{L}-\hat{M}+\hat{\mu}_{\hat{M}}^{(2)}),$ where all parts are distinct and in the color $2^{k-1}.$

Remember that $\hat{\lambda}^{(2)}$ has $\hat{L}-\hat{M}$ parts. The largest part from the remaining $\hat{L}-\hat{M}$ parts in $\hat{\nu}$ is at most $\hat{L}-\hat{M}-1$. Hence, those parts can be added as a generalized staircase to $\hat{\lambda}^{(2)},$ following the method used in Step 4. This gives $\hat{\lambda}^{(1)}.$

\begin{example}
Consider the partitions from the last example:
\[
\hat{\lambda}^{(2)}=(10_3,8_5,7_3,5_1,4_5,1_2,1_2,1_1),\ \hat{\mu}^{(2)}=(6_4,2_4,2_4,1_4)\ \text{and}\ \hat{\nu}=(11,10,9,8,7,6,3,1,0).
\]
The number of parts in $\hat{\lambda}$ was $\hat{L}=12$ and the number of parts in $\hat{\mu}^{(2)}$ is $\hat{M}=4.$ Hence,
\begin{align*}
(\hat{\mu}^{(2)},\hat{\nu}) &= ((6_4,2_4,2_4,1_4),({11},{10},{9},{8},7,6,3,1,0)) \\
&\mapsto (({17}_4,{12}_4,{11}_4,{9}_4),(7,6,3,1,0)) \\
&= (\hat{\mu}^{(1)},\hat{\nu}^{(1)}).
\end{align*}
Add $\hat{\nu}^{(1)}$ to $\hat{\lambda}^{(2)}$ as a generalized staircase:
\begin{align*}
(\hat{\lambda}^{(2)},\hat{\nu}^{(1)}) &= ((10_3,8_5,7_3,5_1,4_5,1_2,1_2,1_1),({7},6,3,1,0)) \\
&\mapsto (({11}_3,{9}_5,{8}_3,{6}_1,{5}_5,{2}_2,{2}_2,{\overline{\textcolor{black}{1}}}_1),({6},3,1,0)) \\
&\mapsto (({12}_3,{10}_5,{9}_3,{7}_1,{6}_5,{3}_2,{\overline{\textcolor{black}{2}}}_2,\overline{1}_1),(3,1,0)) \\
&\dots \\
&\mapsto (\overline{14}_3,\overline{11}_5,10_3,\overline{7}_1,6_5,3_2,\overline{2}_2,\overline{1}_1) \\
&= \hat{\lambda}^{(1)}.
\end{align*}
\end{example}
\begin{lemma}\label{lemma:diffcondstep7}
The difference condition satisfied by $\hat{\lambda}^{(1)}$ is
    \begin{equation}\label{eq:diffcondstep7}
    \hat{\lambda}_i^{(1)}-\hat{\lambda}_{i+1}^{(1)} \geq \omega(\hat{c}_i^{(1)})+\delta^*(\hat{c}_i^{(1)},\hat{c}_{i+1}^{(1)})-\mathbf{1}_{\hat{\lambda}^{(1)}_{i+1} \text{ is not overlined}},
    \end{equation}
for $i=1,\dots,\hat{L}-\hat{M}-1.$
\end{lemma}
\begin{proof}
    Indeed, if $\hat{\lambda}_{i+1}^{(1)}$ is overlined, then
    \begin{align*}
    \hat{\lambda}_i^{(1)}-\hat{\lambda}_{i+1}^{(1)}
    &=\hat{\lambda}_i^{(2)}+1-\hat{\lambda}_{i+1}^{(2)} \\
    &\stackrel{\eqref{eq:diffcondstep6}}{\geq}\omega(\hat{c}_i^{(2)})+\delta^*(\hat{c}_i^{(2)},\hat{c}_{i+1}^{(2)}) \\
    &=\omega(\hat{c}_i^{(1)})+\delta^*(\hat{c}_i^{(1)},\hat{c}_{i+1}^{(1)})-\mathbf{1}_{\hat{\lambda}_{i+1}^{(1)} \text{ is not overlined}},
    \end{align*}
    and if $\hat{\lambda}_{i+1}^{(1)}$ is not overlined, then
    \begin{align*}
    \hat{\lambda}_i^{(1)}-\hat{\lambda}_{i+1}^{(1)}
    &=\hat{\lambda}_i^{(2)}-\hat{\lambda}_{i+1}^{(2)} \\
    &\stackrel{\eqref{eq:diffcondstep6}}{\geq}\omega(\hat{c}_i^{(2)})+\delta^*(\hat{c}_i^{(2)},\hat{c}_{i+1}^{(2)})-1 \\
    &=\omega(\hat{c}_i^{(1)})+\delta^*(\hat{c}_i^{(1)},\hat{c}_{i+1}^{(1)})-\mathbf{1}_{\hat{\lambda}_{i+1}^{(1)} \text{ is not overlined}}.
    \end{align*}
\end{proof}
\begin{lemma}\label{lemma:smallestpart7}
The smallest part of $\hat{\lambda}^{(1)}$ satisfies
\begin{equation}\label{eq:smallestpart7}
\hat{\lambda}_{\hat{L}-\hat{M}}^{(1)}\geq\omega(\hat{c}_{\hat{L}-\hat{M}}^{(1)}).
\end{equation}
\end{lemma}
\begin{proof}
When we add a generalized staircase to a partition, no changes in size or color are made to the smallest part, so that
\[
\hat{\lambda}_{\hat{L}-\hat{M}}^{(1)}=\hat{\lambda}^{(2)}_{\hat{L}-\hat{M}} \overset{\eqref{eq:smallestpart7}}{\geq} \omega(\hat{c}^{(2)}_{\hat{L}-\hat{M}})=\omega(\hat{c}_{\hat{L}-\hat{M}}^{(1)}).
\]
\end{proof}

\addtocounter{enumi}{-2}\item The process in Step 1 is actually very similar to the process of adding a generalized staircase to a partition. Indeed, we consider $\mu^{(1)}$, a partition into distinct parts in color $2^{k-1}$, instead of a generalized staircase. For each part $p$ in $\mu^{(1)}$ which is smaller than $L$, we add $1$ to the first $p$ part of $\lambda^{(1)}$, but instead of overlining the $(p+1)$-th part, we add $2^{k-1}$ to the color of the $p$-th part. Step 1 is therefore reversible, as is the operation of adding a staircase, since we can easily identify where the changes occur by looking at the parts whose color contains the primary color $2^{k-1}$. 

In fact, the inverse operation for Step 1 is very similar to the action of removing a generalized staircase: we identify the parts in $\hat{\lambda}^{(1)}$ whose color's binary decomposition contains $2^{k-1}$ and for each of those parts starting from the smallest one, we take away $2^{k-1}$ from its color and remove $1$ from its value and from each part above it. Those $1$'s, along with the color $2^{k-1},$ are moved to the end of $\hat{\mu}^{(1)},$ as a new part of the color $2^{k-1}.$ This is always possible since $\hat{\lambda}^{(1)}$ has $\hat{L}-\hat{M}$ parts and the smallest part of $\hat{\mu}^{(1)}$ is greater than $\hat{L}-\hat{M}.$ The two partitions obtained at the end of this step are called $\hat{\lambda}$ and $\hat{\mu}.$

\begin{example}
Consider the partitions $\hat{\lambda}^{(1)}$ and $\hat{\mu}^{(1)}$ obtained at the end of the previous example, then find the parts in $\hat{\lambda}^{(1)}$ whose color's binary decomposition contains $4$ and apply the procedure of Step $\hat{1}$:
\begin{align*}
(\hat{\lambda}^{(1)},\hat{\mu}^{(1)}) &= ((\overline{14}_3,\overline{11}_5,10_3,\overline{7}_1,6_{{5}},3_2,\overline{2}_2,\overline{1}_1),(17_4,12_4,11_4,9_4)) \\
&\mapsto ((\overline{{13}}_3,\overline{{10}}_5,{9}_3,\overline{{6}}_1,{5_1},3_2,\overline{2}_2,\overline{1}_1),(17_4,12_4,11_4,9_4,{5_4})) \\
&\mapsto ((\overline{{12}}_3,\overline{{9}}_{{1}},9_3,\overline{6}_1,5_1,3_2,\overline{2}_2,\overline{1}_1),(17_4,12_4,11_4,9_4,5_4,{2_4})) \\
&= (\hat{\lambda},\hat{\mu}).
\end{align*}
The overpartition $\hat{\lambda}$ and the partition $\hat{\mu}$ are exactly the ones we started with at the beginning of the proof, i.e $\hat{\lambda}=\lambda$ and $\hat{\mu}=\mu.$
\end{example}

\begin{lemma}\label{lemma:diffcondstep8}
    The difference condition satisfied by $\hat{\lambda}$ is
    \begin{equation}\label{eq:diffcondstep8}
    \hat{\lambda}_i-\hat{\lambda}_{i+1} \geq \omega(\hat{c}_i)+\delta(\hat{c}_i,\hat{c}_{i+1})-\mathbf{1}_{\hat{\lambda}_{i+1} \text{ is not overlined}},
    \end{equation}
    for $i=1,\dots,\hat{L}-\hat{M}-1.$
\end{lemma}
\begin{proof}
The inequality follows from Lemma~\ref{lemma:diffcondstep7} and Table~\ref{table2}.
\begin{table}[H]
    \centering
    \begin{tabular}{|c|c|c|}
    \hline
      & \textbf{$\hat{c}_i=\hat{c}_i^{(1)}$} & \textbf{$\hat{c}_i=\hat{c}_i^{(1)}-2^{k-1}$} \\
    \hline
    $\hat{\lambda}_{i+1}^{(1)}$ is overlined & 
    \begin{minipage}{5cm}
    \vspace{6pt} 
    \begin{itemize}[left=0pt]
        \item $\hat{\lambda}_i-\hat{\lambda}_{i+1}=\hat{\lambda}_i^{(1)}-\hat{\lambda}_{i+1}^{(1)}$
        \item $\omega(\hat{c}_i)=\omega(\hat{c}_i^{(1)})$
        \item $\delta(\hat{c}_i,\hat{c}_{i+1}) = \delta^*(\hat{c}_i^{(1)},\hat{c}_{i+1}^{(1)})$
        \item $\mathbf{1}_{\hat{\lambda}_{i+1} \text{ is not overlined}} = 0$
    \end{itemize}
    \vspace{1pt} 
    \end{minipage}
    & \begin{minipage}{5cm}
    \vspace{6pt} 
    \begin{itemize}[left=0pt]
        \item $\hat{\lambda}_i-\hat{\lambda}_{i+1}= \hat{\lambda}_i^{(1)}-\hat{\lambda}_{i+1}^{(1)}-1$
        \item $\omega(\hat{c}_i)=\omega(\hat{c}_i^{(1)})-1$
        \item $\delta(\hat{c}_i,\hat{c}_{i+1}) = \delta^*(\hat{c}_i^{(1)},\hat{c}_{i+1}^{(1)})$
        \item $\mathbf{1}_{\hat{\lambda}_{i+1} \text{ is not overlined}} = 0$
    \end{itemize}
    \vspace{1pt} 
    \end{minipage} \\
    \hline
    $\lambda_{i+1}^{(1)}$ is not overlined & \begin{minipage}{5cm}
    \vspace{6pt} 
    \begin{itemize}[left=0pt]
        \item $\hat{\lambda}_i-\hat{\lambda}_{i+1}=\hat{\lambda}_i^{(1)}-\hat{\lambda}_{i+1}^{(1)}$
        \item $\omega(\hat{c}_i)=\omega(\hat{c}_i^{(1)})$
        \item $\delta(\hat{c}_i,\hat{c}_{i+1}) = \delta^*(\hat{c}_i^{(1)},\hat{c}_{i+1}^{(1)})$
        \item $\mathbf{1}_{\hat{\lambda}_{i+1} \text{ is not overlined}} = 1$
    \end{itemize}
    \vspace{1pt} 
    \end{minipage} & \begin{minipage}{5cm}
    \vspace{6pt} 
    \begin{itemize}[left=0pt]
        \item $\hat{\lambda}_i-\hat{\lambda}_{i+1}= \hat{\lambda}_i^{(1)}-\hat{\lambda}_{i+1}^{(1)}-1$
        \item $\omega(\hat{c}_i)=\omega(\hat{c}_i^{(1)})-1$
        \item $\delta(\hat{c}_i,\hat{c}_{i+1}) = \delta^*(\hat{c}_i^{(1)},\hat{c}_{i+1}^{(1)})$
        \item $\mathbf{1}_{\hat{\lambda}_{i+1} \text{ is not overlined}} = 1$
    \end{itemize}
    \vspace{1pt} 
    \end{minipage} \\
    \hline
    \end{tabular}
\caption{Note that the values in this table are the same whether $\hat{c}_{i+1}=\hat{c}_{i+1}^{(1)}-2^{k-1} \text{ or } \hat{c}_{i+1}=\hat{c}_{i+1}^{(1)}.$}\label{table2}
\end{table}
\end{proof}
\begin{lemma}\label{lemma:smallestpart8}
The smallest part of $\hat{\lambda}$ satisfies
\begin{equation}\label{eq:smallestpart8}
\hat{\lambda}_{\hat{L}-\hat{M}}\geq\omega(\hat{c}_{\hat{L}-\hat{M}}).
\end{equation}
\end{lemma}
\begin{proof}
In Step $\hat{1}$, either $\hat{\lambda}_{\hat{L}-\hat{M}}^{(1)}$ is not changed at all by the procedure, and in this case we have
\[
\hat{\lambda}_{\hat{L}-\hat{M}}=\hat{\lambda}_{\hat{L}-\hat{M}}^{(1)}\overset{\eqref{eq:smallestpart7}}{\geq}\omega(\hat{c}_{\hat{L}-\hat{M}}^{(1)})=\omega(\hat{c}_{\hat{L}-\hat{M}}),
\]
or it is changed and we have
\[
\hat{c}_{\hat{L}-\hat{M}}=\hat{c}_{\hat{L}-\hat{M}}^{(1)}-2^{k-1} \implies \hat{\lambda}_{\hat{L}-\hat{M}}=\hat{\lambda}_{\hat{L}-\hat{M}}^{(1)}-1\overset{\eqref{eq:smallestpart7}}{\geq}\omega(\hat{c}_{\hat{L}-\hat{M}}^{(1)})-1=\omega(\hat{c}_{\hat{L}-\hat{M}}).
\]
\end{proof}
\end{enumerate}

The inverse procedure is complete. A final check is required to confirm that the resulting partitions are in the correct form.

\begin{proposition}
Let $\tilde{n} \coloneqq n-\sum_{j=1}^{x_k} \hat{\mu}_j.$ The partition $\hat{\mu}$ is a partition of $n-\tilde{n}$ into $x_k$ distinct parts each in the color $2^{k-1}$ and
\[
\hat{\lambda}\in\overline{S}(x_1,\dots,x_{k-1};m,\tilde{n}).
\]
\end{proposition}
\begin{proof}
The partition $\hat{\mu}$ has exactly $x_k$ parts, since there were $x_k$ parts in $\hat{\lambda}^{(4)}$ whose color contained the primary color $2^{k-1}$ and they each contributed a part of color $2^{k-1}$ in $\hat{\mu}$. Those parts are all distinct, since $\hat{\mu}^{(1)}$ had distinct parts and we only added new distinct parts to it to create $\hat{\mu}$ in Step $\hat{1}$.

We now verify all the conditions for $\hat{\lambda}$:
\begin{enumerate}
    \item \hl{It is an overpartition of $\tilde{n}$:}
    \;It is an overpartition by construction. Since all parts in $\hat{\lambda}$ and $\hat{\mu}$ come from the original $\hat{\lambda}^{(4)}$ counted by $\overline{S}(x_1,\dots,x_k;m,n),$
    \[
    \sum_{i=1}^{\hat{L}-\hat{M}}\lambda_i=n-\sum_{j=1}^{x_k}\hat{\mu}_j=\tilde{n}.
    \]
    \item \hl{It is in $2^{k-1}-1$ colors:}
    \;The overpartition $\hat{\lambda}$ is made of only $2^{k-1}-1$ colors since all the parts whose color contained $2^{k-1}$ were moved to $\hat{\mu}.$
    \item \hl{It has $m$ non-overlined parts:}
    \;To understand why, let us analyze Steps $\hat{4}$ and $\hat{2}$. First, note that $\hat{\lambda}^{(4)}$ had exactly $m$ non-overlined parts by hypothesis. In Step $\hat{4}$, the generalized staircase formed from $\hat{\lambda}^{(4)}$ had $\hat{L}-m$ parts (one for each overlined part in $\hat{\lambda}^{(4)}$). After Step $\hat{3}$, the partition $\hat{\lambda}^{(2)}$ had $\hat{L}-\hat{M}$ parts (since $\hat{M}$ is the number of parts in $\hat{\mu}^{(2)}).$ So, when the $(\hat{L}-m)-\hat{M}$ smallest parts in $\hat{\nu}$ are added back to $\hat{\lambda}^{(2)}$ in Step $\hat{2}$, the number of parts left non-overlined is exactly $(\hat{L}-\hat{M})-((\hat{L}-m)-\hat{M})=m.$ Step $\hat{1}$ does not change this. So $\hat{\lambda}$ has $m$ non-overlined parts.
    \item \hl{The smallest part satisfies $\hat{\lambda}_{\hat{L}-\hat{M}} \geq \omega(\hat{c}_{\hat{L}-\hat{M}})$:}
	\;This follows from Lemma~\ref{lemma:smallestpart8}.
    \item \hl{There are $x_i$ parts that have $2^{i-1}$ in their color's binary decomposition, for $i=1,\dots,k-1$:}
    \;This is true for $\hat{\lambda}^{(4)},$ and the number of appearances of each power of $2$ has not changed during the four steps, though they may have moved from a partition to another. Since the parts in $\hat{\mu}$ are colored with the $x_k$ appearances of $2^{k-1}$ and nothing else, then all powers from $1$ to $2^{k-2}$ are in $\hat{\lambda}$ and, by construction, there are no parts in $\hat{\lambda}$ where the same primary color is used twice for the same part.
    \item \hl{For $i=1,\dots,\hat{L}-\hat{M}-1$, $\hat{\lambda}_i-\hat{\lambda}_{i+1} \geq \omega(\hat{c}_i) + \delta(\hat{c}_i,\hat{c}_{i+1}) - \mathbf{1}_{\hat{\lambda}_{i+1} \text{ is not overlined}}$:}
    \;This follows from Lemma~\ref{lemma:diffcondstep8}.
    \item  \hl{The $\hat{s}$ smallest parts are overlined, where $\hat{s}\coloneqq\sum_{r=1}^{k-1} \hat{V}_{2^r}$:}
    \;Since the procedure in Step $\hat{1}$ only removes $2^{k-1}$ from the color of some of the parts in $\hat{\lambda}^{(1)},$ and none of the parts are colored with only $2^{k-1},$ the smallest power of $2$ in the binary decomposition of the color of each part remains the same throughout the procedure. Therefore, $\hat{s}=\hat{s}^{(1)},$ and since $\hat{s}^{(1)}\coloneqq\sum_{r=1}^{k-2} \hat{V}^{(1)}_{2^r} = \sum_{r=1}^{k-2} \hat{V}^{(2)}_{2^r} = \hat{s}^{(4)} - \hat{M}$ (see Remark~\ref{rem:condition4} for the last equality), the $\hat{s}^{(1)}$ smallest parts in $\hat{\lambda}^{(1)}$ are overlined. 
\end{enumerate}
\end{proof}

\section{Conclusion}\label{sec:conclusion}

In this paper, an interpretation of the infinite product
\[
\frac{(-y_1 q;q)_\infty \cdots (-y_k q;q)_\infty}{(y_1 d q;q)_\infty},
\]
is provided as a generating function for a class of colored overpartitions.

In addition to the result proved in this work, further generalizations naturally suggest themselves. For instance, the following conjectural identity appears to extend the framework in a meaningful way:
\begin{definition}\label{def:Sj}
Let $x_1,\dots,x_k,m,n\in\mathbb{N}$ and $j\in\{1,\dots,k\}.$ We define $\overline{S_j}(x_1,\dots,x_k;m,n)$ to be the same as $\overline{S}(x_1,\dots,x_k;m,n)$ from Definition~\ref{def:Soverlined} except for the fourth condition which is replaced by:
\begin{enumerate}
   \item[$(\widetilde{iv})$] the $s$ smallest parts are overlined, where
\[
s \coloneqq \sum_{\substack{r=0\\r \neq j-1}}^{k-1} V_{2^r},
\]
and $V_{2^r}$ is the number of parts $\lambda_i$ such that $v(c_i)=2^r.$
\end{enumerate}
\end{definition}
\begin{conjecture}
In the setting of Definition~\ref{def:Sj}, the following identity holds:
\[
\sum_{\substack{x_1, \dots, x_k, m, n \geq 0}} \overline{S_j}(x_1,\dots,x_k;m,n) y_1^{x_1} \cdots y_k^{x_k} d^m q^n = \frac{(-y_1 q;q)_\infty \cdots (-y_k q;q)_\infty}{(y_j d q;q)_\infty}.
\]
\end{conjecture}
\smallskip
Another natural direction is to seek an iterative-bijective proof of the following theorem:
\smallskip
\begin{definition}\label{def:Toverlined}
Let $x_1,\dots,x_k,m,n\in\mathbb{N}$ and $j\in\{1,\dots,k\}.$ We define $\overline{T}(x_1,\dots,x_k;m,n)$ to be the same as $\overline{S}(x_1,\dots,x_k;m,n)$ from Definition~\ref{def:Soverlined}, however \underline{without} condition \textit{(iv)}.
\end{definition}
\begin{theorem}[Dousse]
In the setting of Definition~\ref{def:Toverlined}, the following identity holds:
\[
\sum_{\substack{x_1, \dots, x_k,\, m,\, n \geq 0}} \overline{T}(x_1,\dots,x_k; m, n)\, y_1^{x_1} \cdots y_k^{x_k} d^m q^n = \frac{(-y_1 q;q)_\infty \cdots (-y_k q;q)_\infty}{(y_1 d q;q)_\infty \cdots (y_k d q;q)_\infty}.
\]
\end{theorem}
This result is a natural generalization of Theorem~\ref{thm:SchurOverpartitions}, Lovejoy's \emph{Schur's theorem for overpartitions}~\cite{Lovejoy2005}, and was previously proved by J.~Dousse~\cite{Dousse2017} using a combination of the method of weighted words and $q$-difference equations techniques. Additionally, the case $k = 2$ was established bijectively by Raghavendra and Padmavathamma~\cite{RagPad2009}, although not through the same type of iterative-bijective argument employed in the proof of Theorem~\ref{thm:LV}.

These identities suggest interesting open problems whose iterative-bijective proofs remain to be discovered. Investigating them further could yield deeper insights into the structure of overpartitions and the mechanisms underlying their generating functions.

\section{Acknowledgements}
This work originates from the author’s master’s thesis at the University of Geneva. The author thanks Professor Jehanne Dousse for suggesting the problem and for her supervision.

\bibliographystyle{plain}
\addcontentsline{toc}{section}{References}
\bibliography{paper-VELENIK-Laure}

\end{document}